\documentclass[final, 11pt]{amsart}

\usepackage[utf8]{inputenc} 
\usepackage{bbm}

\usepackage{siunitx}
\usepackage{mathrsfs} 

\usepackage{graphicx} 

\usepackage[margin=3cm]{geometry}

\usepackage{verbatim}
\usepackage{amssymb}
\usepackage{amsthm}
\usepackage{amsmath}
\usepackage{amsfonts}
\usepackage{latexsym}
\usepackage{mathtools}
\usepackage{enumitem} 
\usepackage{cite}
\usepackage[english]{babel}

\usepackage[usenames,dvipsnames]{xcolor} 
\usepackage[colorlinks=true, pdfstartview=FitV, linkcolor=blue, citecolor=blue, urlcolor=blue]{hyperref}

\usepackage[color]{showkeys}

\usepackage{color}

\theoremstyle{plain}
\newtheorem{thm}{Theorem}
\newtheorem{remark}[thm]{Remark}
\newtheorem{prop}[thm]{Proposition}
\newtheorem{cor}[thm]{Corollary}
\newtheorem{lemma}[thm]{Lemma}

\numberwithin{equation}{section}
\numberwithin{thm}{section}

\newcommand{\pv}{\text{pv}}

\def\R {\mathbb R}
\def\T {\mathbb T}

\newcommand{\fsonenu}{\dot{\mathcal{F}}^{s,1}_{\nu}}
\newcommand{\fsplusonenu}{\dot{\mathcal{F}}^{s+1,1}_{\nu}}
\newcommand{\fzerone}{\mathcal{F}^{0,1}}
\newcommand{\fzeronenu}{\mathcal{F}^{0,1}_{\nu}}

\newcommand{\foneonenu}{\dot{\mathcal{F}}^{1,1}_{\nu}}

\newcommand{\ftwoonenu}{\dot{\mathcal{F}}^{2,1}_{\nu}}

\newcommand{\ffourone}{\dot{\mathcal{F}}^{4,1}}

\begin{document}

\addtocounter{footnote}{-1}\let\thefootnote\svthefootnote	

\title[On Nonlinear Stability of Muskat Bubbles]{On Nonlinear Stability of Muskat Bubbles}


\author[F. Gancedo]{Francisco Gancedo}
\address{Departamento de An\'alisis Matem\'atico, Universidad de Sevilla, C/Tarfia s/n, Campus Reina Mercedes, 41012, Sevilla, Spain. \href{mailto:fgancedo@us.es}{fgancedo@us.es}}

\author[E. Garc\'ia-Ju\'arez]{Eduardo Garc\'ia-Ju\'arez}
\address{Departamento de An\'alisis Matem\'atico, Universidad de Sevilla, C/Tarfia s/n, Campus Reina Mercedes, 41012, Sevilla, Spain. \href{mailto:egarcia12@us.es}{egarcia12@us.es}}

\author[N. Patel]{Neel Patel}
\address{Department of Mathematics and Statistics, University of Maine, Orono. Neville Hall 04469, Orono, Maine, USA 48109. \href{mailto:neel.patel@maine.edu}{neel.patel@maine.edu}}

\author[R. M. Strain]{Robert M. Strain}
\address{Department of Mathematics, University of Pennsylvania, David Rittenhouse Lab., 209 South 33rd St., Philadelphia, PA 19104, USA. \href{mailto:strain@math.upenn.edu}{strain@math.upenn.edu}}

\date{\today}

\begin{abstract}
In this paper we consider gravity-capillarity Muskat bubbles in 2D. We obtain a new approach to improve our result in \cite{GancedoGarciaJuarezPatelStrain23}. Due to a new bubble-adapted formulation, the improvement is two fold. We significantly condense the proof and we now obtain the global well-posedness result for Muskat bubbles in critical regularity.

\end{abstract}

\setcounter{tocdepth}{3}

\maketitle
\tableofcontents

\section{Introduction}

In this paper, we consider the dynamics of a closed curve interface between two incompressible and immiscible fluids in a two-dimensional porous media.
The physical principle for porous media flow velocity $u(t,x)$ is given by Darcy's law \cite{Darcy56},
\begin{equation}\label{Darcy}
\frac{\mu(x,t)}{\kappa} u(x,t)=-\nabla p(x,t)-(0,g\rho(x,t)).
\end{equation}
Above $\mu$ is the viscosity of the fluid, $\kappa$ is the porous medium permeability, $p$ is the fluid pressure, $g$ is the gravity constant and $\rho$ is the fluid density. Darcy's law is paired with the incompressibility condition that the fluid velocity is divergence free:
\begin{equation}\label{incom}
\nabla \cdot u(x,t)=0.
\end{equation}
The two-phase problem of two immiscible fluids of different densities and viscosities filtered in the porous medium can be described as:
\begin{equation}\label{patchsolution}
\mu(x,t)=\left\{\begin{array}{rl}
\mu^1,& x\in D(t),\\
\mu^2,& x\in \mathbb{R}^2\smallsetminus \overline{D(t)},
\end{array}\right.
\quad 
\rho(x,t)=\left\{\begin{array}{rl}
\rho^1,& x\in D(t),\\
\rho^2,& x\in \mathbb{R}^2\smallsetminus \overline{D(t)}.
\end{array}\right.
\end{equation}
The density $\rho^{i}$ and viscosity $\mu^{i}$ are different non-negative constant values in complementary domains. In this paper, the domain $D(t)$ will be a bounded and connected domain, i.e. a bubble, and hence, the interface $\partial D(t)$ is a closed curve.
To complement the gravitational force, the interface may be subjected to surface tension:
\begin{equation}
\label{surfacetension}
p^2(x)-p^1(x)=-\sigma K(x),\hspace{1cm}x\in\partial D(t),
\end{equation}
where $p^1$, $p^2$ are the limits of the pressure from the interior and the exterior of $D(t)$, $K$ is the boundary curvature, taken to be positive for circles that are counterclockwise parameterized, and $\sigma>0$ is the surface tension coefficient. Finally, the free boundary $\partial D(t)$ is transported by the fluid velocity.
The problem of studying the dynamics of this situation is called the Muskat problem \cite{Muskat34}. 

Notably, gravity $g$ in \eqref{Darcy} is an unstable force for any closed curve interface separating fluids of different densities, which is the situation we deal with in this paper. This instability is seen from the sign of the Rayleigh-Taylor coefficient in the portion of the curve where the denser fluid lays above the lighter fluid and where the less viscous fluid moves into the more viscous one. Indeed, in the thoroughly studied regime where the fluid interface is an infinite curve, local well-posedness depends on the sign of the Rayleigh-Taylor coefficient, see e.g. \cite{CCG11}. When the infinite interface can be described by a graph, the correct sign is given by the denser fluid filling the domain below the graph, see e.g. \cite{CG07}. For a bubble, the Rayleigh-Taylor condition has mean zero, and therefore it has an unstable structure which is addressed in this paper.

Equations \eqref{Darcy}-\eqref{surfacetension} will imply the contour evolution equation we derive below.

\subsection{Contour equation}
We will now derive the evolution equation for the fluid interface of a bubble. The boundary $\partial D(t)$ represents the fluid interface of a drop or bubble and is parameterized counterclockwise by 
$$
\partial D(t)=\{z(\alpha,t):\, \alpha\in\mathbb{T}=\mathbb{R}/2\pi\}.
$$
Darcy's law \eqref{Darcy} implies that the fluid vorticity is concentrated on the moving boundary:
$$
\langle
\nabla^{\bot}\cdot u(t),v
\rangle =\int_\T\omega(\alpha,t)\varphi(z(\alpha,t))d\alpha.
$$
Above, $\omega$ is the vorticity amplitude and $\varphi\in C^{\infty}_c(\R^2)$ is a test function. Then the Biot-Savart law provides an expression for the fluid velocity:
\begin{equation*}
u(x,t)=\frac1{2\pi} \int_{\mathbb{T}} \frac{(x-z(\alpha,t))^\perp}{|x-z(\alpha,t)|^2}\omega(\alpha,t)d\alpha,\hspace{1cm}x\neq z(\alpha,t).
\end{equation*}
Taking limits towards the boundary, the identity above provides
\begin{equation*}
\begin{aligned}
u^1(z(\alpha,t))&=BR[z](\omega)(\alpha,t)-\frac{\omega(\alpha,t)}2\frac{\partial_{\alpha}z(\alpha,t)}{|\partial_{\alpha}z(\alpha,t)|^2},\\
u^2(z(\alpha,t))&=BR[z](\omega)(\alpha,t)+\frac{\omega(\alpha,t)}2\frac{\partial_{\alpha}z(\alpha,t)}{|\partial_{\alpha}z(\alpha,t)|^2},
\end{aligned}
\end{equation*}
where
\begin{equation}\label{eqBR}
BR[z](\omega)(\alpha,t)=\frac1{2\pi} \pv\int_{\mathbb{T}} \frac{(z(\alpha,t)-z(\beta,t))^\perp}{|z(\alpha,t)-z(\beta,t)|^2}\omega(\beta,t)d\beta.
\end{equation}
The limits above are used in Darcy's law to obtain
\begin{equation}\label{Darcyboundary}
\frac{\mu^j}{\kappa}u^j(z(\alpha,t))=-\nabla p^j(z(\alpha,t),t)-(0,g\rho^j).
\end{equation}
After applying a dot product in the tangential direction to both boundary limit identities, take their difference and use \eqref{surfacetension} and \eqref{Darcyboundary} to obtain
\begin{equation}\label{eqOmega}
\omega(\alpha,t)=2 A_\mu \mathcal{D}[z](\omega)(\alpha,t)+2A_\sigma\partial_{\alpha}K(\alpha,t)-2A_\rho\partial_{\alpha}z_2(\alpha,t),
\end{equation}
where
\begin{equation}\label{Dz}
\begin{aligned}
\mathcal{D}[z](\omega)(\alpha,t)&=-BR[z](\omega)(\alpha,t)\cdot\partial_{\alpha}z(\alpha,t)\\
&=\frac1{2\pi} \pv\int_{\mathbb{T}}
\frac{(z(\alpha,t)\!-\!z(\beta,t))\cdot\partial_{\alpha}z(\alpha,t)^\perp }{|z(\alpha,t)-z(\beta,t)|^2}\omega(\beta,t)d\beta,
\end{aligned}
\end{equation}
\begin{equation}\label{AmuAsigmaArho}
A_\mu=\frac{\mu^2-\mu^1}{\mu^2+\mu^1},\hspace{1cm} A_{\sigma}=\frac{\kappa \sigma}{\mu^2+\mu^1}, \hspace{1cm}A_\rho=\frac{g\kappa(\rho^2-\rho^1)}{\mu^2+\mu^1},
\end{equation}
and the curvature is given by
\begin{equation}\label{curvature}
K(\alpha,t)=\frac{\partial_{\alpha}z(\alpha,t)^\perp\cdot\partial_{\alpha}^2z(\alpha,t)}{|\partial_{\alpha}z(\alpha,t)|^3}.
\end{equation}
We notice that there are two dimensionless numbers in the problem:
\begin{equation}\label{AmuAsigmaArho}
A_\mu=\frac{\mu^2-\mu^1}{\mu^2+\mu^1},\hspace{1cm} A_{\rho\sigma}=\frac{A_\rho R^2}{A_\sigma}=\frac{g(\rho^2-\rho^1)R^2}{\sigma}.
\end{equation}
The kinematic condition in the normal direction finally yields
\begin{equation}\label{eqcurve}
z_t(\alpha,t)\cdot\partial_{\alpha}z(\alpha,t)^\perp=BR[z(\alpha,t)](\omega(\alpha,t))\cdot\partial_{\alpha}z(\alpha,t)^\perp.
\end{equation}
The Muskat problem consists in finding a solution $(z,\omega)$ to the equations \eqref{eqOmega} and \eqref{eqcurve} given the initial datum $D_0=D(0)$.

It is easy to see that \eqref{Darcy}, \eqref{incom} and \eqref{patchsolution} provide a potential velocity
$$
u(x,t)=\nabla\phi(x,t),\quad x\not\in\partial D(t),
$$
with $\phi$ harmonic in $\R^2\smallsetminus\partial D(t)$. In the one-fluid case ($\mu^j=0=\rho^j$, $j=1$ or 2), the function $\phi$ is obtained in the fluid region
$$
\left\{\begin{array}{rcl}
\Delta\phi(x,t)&=&0,\quad x\mbox{ belong to the fluid region},\\
\frac{\mu^l}{\kappa}\phi(x,t)&=&-\sigma K(x)-g\rho^lx_2,\quad x\in\partial D(t),\, l\neq j.
\end{array}\right.
$$
The double layer potential theory provides $\phi(x,t)$ in the dry region. In the two-fluid case we solve
$$
\left\{\begin{array}{rcl}
\Delta\phi(x,t)&=& 0,\, x\in D(t)\cup \mathbb{R}^2\smallsetminus \overline{D(t)} ,\\
\frac{\mu^2}{\kappa}\phi^2(x,t)-\frac{\mu^1}{\kappa}\phi^1(x,t)&=&-\sigma K(x)-g(\rho^2-\rho^1)x_2,\, x\in\partial D(t),
\end{array}\right.
$$
with $\phi^j$ defined as before.
Defining the difference of the potentials at the moving boundary
$$
\Omega(\alpha,t)=\phi^2(z(\alpha,t),t)-\phi^1(z(\alpha,t),t),
$$
we find
\begin{equation}\label{vorti_pot}
    \partial_{\alpha}\Omega(\alpha,t)=(u^2(z(\alpha,t),t)-u^1(z(\alpha,t),t))\cdot\partial_{\alpha}z(\alpha,t)=\omega(\alpha,t).
\end{equation}
This fact implies that the amplitude of the vorticity has mean zero,
\begin{equation}\label{zeromean}
\int_{\mathbb{T}}\omega(\alpha)d\alpha=0,
\end{equation}
and therefore the system has finite energy
$$
\int_{\R^2}\rho(x,t)|u(x,t)|^2dx<C.
$$

\subsection{Main result}

We consider initial configurations where one fluid forms a bubble enclosed by the other. More precisely, we consider initial domains $D_0\subset\mathbb{R}^2$ whose boundary $\partial D_0$ admits a polar parametrization with pole at the centroid of $D_0$. That is, if we denote the center of mass of $D_0$ by
\begin{equation*}
    \begin{aligned}  
    x_{c,0}&=\frac{1}{|D_0|}\int_{D_0}x \,dx,\qquad |D_0|=\text{Vol}(D_0)=\int_{D_0}dx,
    \end{aligned}
\end{equation*}
we assume that there exists a parametrization of $\partial D_0$,     $\partial D_0=\{z_0(\alpha):\, \alpha\in \mathbb{T}\}$,
that can be written in the following form
\begin{equation}\label{z0_polar}
    \begin{aligned} z_0(\alpha)&=x_{c,0}+R(1+f_0(\alpha))\big(\cos(\alpha),\sin(\alpha)\big),
    \end{aligned}
\end{equation}
where $R=\sqrt{|D_0|/\pi}$ and $f_0:\mathbb{T}\to (-1,\infty)$ is such that
\begin{equation}\label{f_cond}
    \begin{aligned}  
    0&=\int_{\mathbb{T}}(1+f_0(\alpha))^3\big(\cos(\alpha),\sin(\alpha)\big)d\alpha,\qquad \pi=\frac12\int_{\mathbb{T}}(1+f_0(\alpha))^2d\alpha.
    \end{aligned}
\end{equation}
The first condition in \eqref{f_cond} ensures that the center of mass of the bubble described in \eqref{z0_polar} is $x_{c,0}$.  The second condition in \eqref{f_cond} guarantees that the volume of the domain enclosed by \eqref{z0_polar} is $|D_0|$. Since the problem is translation-invariant, we can, without loss of generality, assume that the initial center of mass is at the origin, $x_{c,0}=0$.
\begin{remark}
We say that two closed, simply connected curves $\Gamma_1,\Gamma_2\subset\mathbb{R}^2$  of class $C^1$ are $C^1$-close of order $\varepsilon$ if the Hausdorff distance between their normal bundles is smaller than $\varepsilon$. Then, the above assumption for $D_0$ holds whenever $D_0$ is a small enough $C^1$-perturbation of a circle with area $|D_0|$. That is, if $D_0$ is $C^1$-close of small enough order to a circle with the same area, then $D_0$ can be parametrized in polar coordinates with origin at its centroid. This easily follows from the existence of a normal parametrization (see for example \cite[Chapter 2]{PrussSimonett2016}).
\end{remark}
\begin{remark}
    Given any $x_{c,0}\in\mathbb{R}^2$, $R>0$, the curve $z_0(\alpha)$ \eqref{z0_polar}, with $f_0$ satisfying \eqref{f_cond}, is a circle if and only if $f_0\equiv0$. In fact, if $z_0$ is a circle and \eqref{f_cond} holds, then $x_{c,0}$ is its center and $R$ its radius, thus   $|z_0(\alpha)-x_{c,0}|=R|1+f_0(\alpha)|=R$, which implies $f_0\equiv0$. 
\end{remark}
We will also use a polar parametrization for the evolving interface, but with a time-dependent pole $c(t)=(c_1(t),c_2(t))$,
\begin{equation}\label{z_param}
    \begin{aligned} z(\alpha,t)&=R(1+f(\alpha,t))\big(\cos(\alpha),\sin(\alpha)\big)+\frac{A_\rho}{A_{\rho\sigma}}c(t),
    \end{aligned}
\end{equation}
with $c(t)$ defined in terms of $f$:
\begin{equation}\label{c_def}
    \begin{aligned}
        c_1(t)&=\frac{1}{\pi}\int_0^t \int_{\mathbb{T}}f(\alpha,s)\cos{(\alpha})d\alpha ds,\\
        c_2(t)&=A_{\rho\sigma} t+\frac{1}{\pi}\int_0^t \int_{\mathbb{T}}f(\alpha,s)\sin{(\alpha})d\alpha ds.
    \end{aligned}
\end{equation}
The specific choice for the definition of $c(t)$ above will become apparent later in the proof. For clarity of exposition, we will keep the notation $c(t)$.
We should notice that the freedom in the choice of the parametrization for the interface translates into the definition of $c(t)$.

To solve the Muskat problem \eqref{eqBR}-\eqref{eqcurve} with $z_0$ \eqref{z0_polar} initial datum, we will need to determine $f(\alpha,t)$ solving the equations \eqref{vorticity_tilde}-\eqref{fequation1} given the initial function $f_0$ satisfying \eqref{f_cond}.

We will work with functional spaces closely related to the Wiener algebra. We recall that the Wiener algebra $\fzerone(\mathbb{T})$ (also denoted $\mathbb{A}(\mathbb{T})$) is the space of functions whose Fourier series is absolutely convergent. We define analogous spaces with fractional regularity. For $s\geq0$,  let $\dot{\mathcal{F}}^{s,1}(\mathbb{T})$ be the homogeneous Wiener spaces induced by the seminorm
\begin{equation}\label{norm}
\|f\|_{\dot{\mathcal{F}}^{s,1}}=\sum_{k\in \mathbb{Z}\setminus\{0\}}|k|^s|\hat{f}(k)|,
\end{equation}
and $\mathcal{F}^{s,1}(\mathbb{T})$ their nonhomogeneous counterpart with norm
\begin{equation}\label{norm}
\|f\|_{\mathcal{F}^{s,1}}=|\hat{f}(0)|+\sum_{k\in \mathbb{Z}\setminus\{0\}}|k|^s|\hat{f}(k)|.
\end{equation}
Similarly, we can define spaces with analytic regularity: for $\nu>0$, $t>0$, we define the space $\dot{\mathcal{F}}^{s,1}_{\nu,t}(\mathbb{T})$ induced by the norm
\begin{equation}\label{normanalytic}
\|f\|_{\dot{\mathcal{F}}^{s,1}_{\nu,t}}=\sum_{k\in \mathbb{Z}\setminus\{0\}}e^{\nu|k|t}|k|^s|\hat{f}(k)|.
\end{equation}
Along the paper, the norm $\dot{\mathcal{F}}^{s,1}_{\nu,t}(\mathbb{T})$ will always be used on a time-dependent function at time $t$, $\|f(\cdot,t)\|_{\dot{\mathcal{F}}^{s,1}_{\nu,t}}$. Therefore, we will omit the sub-index $t$ and also the domain $\mathbb{T}$, and simply denote $\dot{\mathcal{F}}^{s,1}_\nu\equiv\dot{\mathcal{F}}^{s,1}_{\nu,t}$.

The main contribution of this paper is to show the global-in-time existence, uniqueness and asymptotic stability of classical solutions to the Muskat problem with initial data given by a perturbation of a circle with $\mathcal{F}^{1,1}$ regularity. 
This regularity is critical in the sense that in both limits $A_\rho=0$ or $A_\sigma=0$, the scaling invariance of the Muskat problem respects the norm.
The solution converges exponentially fast to a circle that moves vertically with constant velocity equal to $A_\rho$. By classical solutions we mean that the equations are satisfied pointwise for all $t>0$ and that the initial data is achieved in the strong sense $\lim_{t\to0^+}\|f(t)-f_0\|_{\mathcal{F}^{1,1}}=0$. Furthermore, we show that the solution is analytic at all positive times.
\begin{thm}[Main Theorem]\label{MainTheorem}
    Let $f_0\in \mathcal{F}^{1,1}(\mathbb{T})$ and $A_\sigma>0$. For any $A_\mu\in[-1,1]$ and $A_{\rho\sigma}\in\mathbb{R}$, there exists $\nu>0$ and  $\varepsilon=\varepsilon(A_\mu, |A_{\rho\sigma}|)>0$ such that, for any $T>0$, if $\|f_0\|_{\mathcal{F}^{1,1}}\leq \varepsilon$, then  there exists a unique classical solution $f\in C([0,T];\mathcal{F}^{1,1}_\nu)\cap L^1(0,T;\mathcal{F}^{4,1}_\nu)$ to \eqref{fequation1}. Moreover, for $t\geq0$,
    \begin{equation*}
        \begin{aligned}
            \|f(t)\|_{\mathcal{F}^{1,1}_\nu}\leq \|f_0\|_{\mathcal{F}^{1,1}}e^{-Ct}, 
        \end{aligned}
    \end{equation*}
    for some positive constant $C=C(\|f_0\|_{\mathcal{F}^{1,1}},|A_{\rho\sigma}|,A_\mu,\nu)$.
\end{thm}
Theorem \ref{MainTheorem} will be a consequence of the Banach fixed point theorem and Theorem \ref{fixed_theo}, since functions belonging to the space $\mathcal{F}^{1,1}_\nu$ are analytic for $t>0$.

\begin{remark} 
    The expression of $\varepsilon(A_\mu, |A_{\rho\sigma}|)$ can be computed explicitly, as was done in \cite{GancedoGarciaJuarezPatelStrain23}. For clarity of exposition and to avoid repetition, we will simply write $\varepsilon$. 
\end{remark}

\subsection{Previous results}

Most of the literature on the two-phase (i.e. two fluids) Muskat problem has focused on the case where the interface is an infinite curve and the fluid domains $D(t)$ and $D(t)^{c}$ are connected and unbounded. In this regime, the sign of the Rayleigh-Taylor coefficient provides a condition for local well-posedness. Moreover, when the infinite curve is a graph, the Rayleigh-Taylor sign condition holds when the denser fluid is below the graph. Local well-posedness of the graphical interface in subcritical spaces (see e.g. \cite{ConstGancShvyVicol2017},  \cite{Matioc19} \cite{AlazardLazar2020}, \cite{NguyenPausader20subcritical}) and, more recently, in critical spaces (see \cite{Alazardendpoint}, \cite{ChenNguyenXu22} and references therein)  is known. Further literature gives other arguments for local well-posedness for large data in certain subcritical spaces and global well-posedness for small data even in critical spaces  (see \cite{ChengRafaShkoller2016}, \cite{CordobaLazar2021}, \cite{Nguyen2022}, \cite{GJGomezSerranoNguyenPausader22}, \cite{GHHaziotGomezSerranoPausader23} and references therein). Interestingly, a graphical interface can turn into a non-graph interface and the Rayleigh-Taylor coefficient changes sign at the turnover point \cite{CCFGL12}. Smooth curve solutions have analytic regularity apart from turnover points \cite{Shi2023}. Also, both types of double shifting between stable and unstable regimes can occur \cite{CordobaGSZlatos2015}, \cite{CordobaGSZlatos2017}.  For initial data of explicit size and for settings with various conditions on the viscosity, permeability and density jumps between the two fluids, global well-posedness and quantitative estimates are proven in \cite{CCGRS-16}, \cite{Cam17}, \cite{GGPS}, \cite{PS17}, \cite{PatelNikhil21}, \cite{AlonsoOranGraneroBelinchon22}. Considering surface tension on an infinite interface, well-posedness is also known in \cite{Ambrose14}, \cite{Nguyen20s}, \cite{PatrickNguyen21}, \cite{MatiocMatioc23}, weak solutions are constructed in \cite{JacobsKimMeszaros21} and instability of fingering patterns is shown in \cite{EscherMatioc2011}.

Surface tension, which is a dominant force in this paper, has also been considered in the one-phase (i.e. one fluid) case. Note that the one-phase Muskat problem has a different structure, where global well-posedness holds even for initial interfaces with large slope \cite{DongGancedoNguyen2023} and local well-posedness is shown even for interfaces with cusps \cite{AgrawalPatelWu2023}. Surface tension stabilizes even a gravity unstable thin one-phase layer, as shown in \cite{GancedoGraneroBelinchonStefano2020}, with higher-order approximations to the model given in \cite{Bocchi23}. Furthermore, in the one-phase problem with surface tension and an external pressure, traveling wave solutions exist \cite{Nguyenlarge23}.

On the other hand, the case of a bubble interface is not as well understood. Mixing (no sharp interface) weak solutions to incompressible porous media flow starting from a bubble interface moved by gravity are demonstrated via convex integration in \cite{CastroFaracoMengual22}. For Hele-Shaw flow without gravity or without a jump in fluid densities, \cite{CP93}, \cite{YT11}, \cite{YT12} prove global well-posedness for small smooth (analytic or $H^{s}$ for $s \geq 4$) perturbations of a circular bubble. In the case of a gravity driven bubble,  our prior result \cite{GancedoGarciaJuarezPatelStrain23} proved the existence of a circular steady state and the global well-posedness for \textit{medium-sized} perturbations of a circular bubble. Importantly, the regularity of the bubble interface allowed for unbounded curvature and the size of the perturbations could be explicitly computed depending on the surface tension, viscosity and density parameters. The results in \cite{GancedoGarciaJuarezPatelStrain23} were based on a tangent vector description of the bubble interface, which already forces one derivative of the boundary into the framework of the evolution equation. Moreover, the tangent vector framework does not depend on the center of mass of the bubble and the linear degeneracy of the $k=\pm 1$ Fourier frequencies of the boundary had to be addressed implicitly. On the other hand, the framework introduced in this paper allows the evolution equation to describe the boundary rather than the tangent to the boundary. With this change, we can prove well-posedness at \textit{critical} regularity. Furthermore, the center of mass can be described explicitly from this new formulation and no implicit argument is needed to control degenerate Fourier modes. For these reasons, the proof is greatly simplified and the main result allows for a larger class of perturbations than in \cite{GancedoGarciaJuarezPatelStrain23}. Moreover, it is important to note that the approach in \cite{GancedoGarciaJuarezPatelStrain23} is specific to a two-dimensional bubble, but the new framework of this paper is general enough to be adapted to the three-dimensional case.

\subsection{Outline}

In the remainder of the introduction, we introduce the notation used in the paper. The first part of Section 2 is devoted to the derivation of the evolution equation in the new parametrization of this paper and its linearization. Then, the linear operator is diagonalized and the incompressibility condition is used to obtain a fully dissipative operator. The main Theorem \ref{MainTheorem} is proved using a fixed point argument. The choice of the parametrization, in particular of $c(t)$ \eqref{c_def}, becomes crucial. Section 3 contains the nonlinear estimates needed to perform the contraction for Theorem \ref{MainTheorem} in the previous section, while the solution to the vorticity integral equation is obtained in Section 4.

\subsection{Notation} We summarize some notations that will be used throughout the paper. We denote by $\mathcal{H}f$ the Hilbert transform
\begin{equation}\label{Hilbert_def}
    \begin{aligned}
        \mathcal{H}f(\alpha):=\frac{1}{\pi}\pv\int_{\mathbb{T}}\frac{f(\alpha-\beta)}{2\tan{(\beta/2)}}d\beta,
    \end{aligned}
\end{equation}
and  its derivative by $|\partial_\alpha|f$,
\begin{equation}\label{lambda_op}
    \begin{aligned}
        |\partial_\alpha|f(\alpha):=\mathcal{H}\partial_\alpha f(\alpha)=\frac{1}{\pi}\pv\int_{\mathbb{T}}\frac{f(\alpha)-f(\alpha-\beta)}{4\sin^2{(\beta/2)}}d\beta.
    \end{aligned}
\end{equation}
Both $\mathcal{F}(f)$ and $\widehat{f}$ will be use to denote the Fourier coefficients of the Fourier series of $f$:
\begin{equation}\label{Fourier_coef}
    \begin{aligned}
        \mathcal{F}(f)(k)\equiv\widehat{f}(k):=\frac{1}{2\pi}\int_{\mathbb{T}}f(\alpha)e^{-ik\alpha}d\alpha.
    \end{aligned}
\end{equation}
We use $n(\alpha)$ and $\tau(\alpha)$ for the unit normal and tangent vectors to the circle centered at the origin 
\begin{equation}\label{unit_vectors}
    n(\alpha)=\big(\cos(\alpha), \sin(\alpha)\big),\quad \tau(\alpha)=\big(-\sin(\alpha), \cos(\alpha)\big)=n(\alpha)^\perp.
\end{equation}

In addition to the Wiener spaces previously introduced, we define, for $\nu\geq0, s\geq0$, the space $Z_\nu^{s,1}$ of functions with zero mean, 
\begin{equation}\label{Z_space}
 Z_\nu^{s,1} = \big\{g\in\mathcal{F}^{s,1}_\nu(\mathbb{T}): \widehat{g}(0) = 0.\big\}
\end{equation}
When $\nu=0, s=0$, we simply denote $Z:=Z_0^{0,1}$.
Furthermore, we denote:
\begin{equation}\label{Xspace}
    \begin{aligned}
        X_{T,\nu}(\mathbb{T})&=C([0,T];Z^{1,1}_\nu)\cap L^1(0,T;\mathcal{F}^{4,1}_\nu),
    \end{aligned}
\end{equation}
and the associated norm
\begin{equation*}
    \begin{aligned}
        \|f\|_{X_{T,\nu}}&=\|f\|_{L^\infty_T(\mathcal{F}^{0,1}_\nu)}+\|f\|_{L^1_T(\mathcal{F}^{4,1}_\nu)}.
    \end{aligned}
\end{equation*}
We will denote $B_{r,T,\nu}(u_0)$ the ball of $X_{T,\nu}$ of radius $r>0$ and centered at $u_0$.
\begin{equation*}
    B_{r,T,\nu}(u_0)=\{f\in X_{T,\nu}(\mathbb{T}): \|f-u_0\|_{X_{T,\nu}}\leq r\}.
\end{equation*}

Given $G(f)$ we denote $G_0(f):=G(0)$ and
\begin{equation*}
    \begin{aligned}
        G_1(f)&:=\lim_{\varepsilon\to 0^+}\frac{d}{d\varepsilon}G(\varepsilon f),\\ 
        G_{\geq1}(f)&:=G(f)-G(0), \\ 
        G_{\geq2}
        (f)&:=G(f)-G(0)-G_1(f).
    \end{aligned}
\end{equation*}
The definition of the notation $G[f] = G_{0}[f]+ G_{1}[f]+G_{\geq 2}[f]$ and $G_{\geq1}[f]$ is equivalent to above.
For convolutions we will use
\begin{equation*}
    f_{1}\ast f_{2} \ast \cdots \ast f_{n} = \ast_{j=1}^{n} f_{j}.
\end{equation*}
Finally, the symbol $\lesssim$ will denote inequality up to a universal multiplicative constant, and we define analogously $\simeq$. Moreover, 
$\lesssim_1$ denotes a bound up to a uniformly bounded and analytic function $C:[0,\epsilon] \rightarrow [0,\infty)$ depending on $\|f\|_{\mathcal{F}^{1,1}_\nu}$ such that $\lim_{x\to0+}C(x)\simeq1$.

\section{Main Theorem}\label{formulation}

We first write the system given by equations \eqref{eqBR}-\eqref{eqcurve} in terms of $f(\alpha, t)$, keeping in mind that $c(t)$ is defined in terms of $f$ \eqref{c_def}. To simplify  notation we shall write $f(\alpha,t)=f(\alpha)$ when there is no danger of confusion. 
We introduce the notation
\begin{equation}\label{deltabeta}
\Delta_{\beta}f(\alpha)=\frac{f(\alpha-\beta)-f(\alpha)}{2\sin{(\beta/2)}},
\end{equation}
and compute the dimensionless curvature $K(f)$ in terms of $f$,
\begin{equation}\label{curvaturef}
K(f)(\alpha)=\frac{-(1+f(\alpha))\partial_\alpha^2f(\alpha)+2(\partial_\alpha f(\alpha))^2+(1+f(\alpha))^2}{\big((\partial_\alpha f(\alpha))^2+(1+f(\alpha))^2\big)^{3/2}}.
\end{equation}
From \eqref{eqOmega}, the vorticity strength is given by
\begin{equation*}
\begin{aligned}
\omega(\alpha)=2 A_\mu \mathcal{D}[f](\omega)(\alpha)+2\frac{A_\sigma}{R}\partial_{\alpha}K(f)(\alpha)-2A_\rho R\partial_\alpha\big((1+f(\alpha))\sin{\alpha}\big),
\end{aligned}
\end{equation*}
where $A_\mu, A_\sigma, A_\rho$ are defined in \eqref{AmuAsigmaArho}.
For convenience, we define $\tilde{\omega}(\alpha)=R\omega(\alpha)/A_\sigma$ so that
\begin{equation}\label{vorticity_tilde}
\tilde{\omega}(\alpha)\!=2 A_\mu \mathcal{D}[f](\tilde{\omega})(\alpha)+2\partial_{\alpha}K(f)(\alpha)-2A_{\rho\sigma} \partial_\alpha\big((1+f(\alpha))\sin{\alpha}\big),
\end{equation}
and
\begin{equation}\label{Doper_f}
\begin{aligned}
\mathcal{D}[f](\tilde{\omega})(\alpha)&\!=\frac{1}{2\pi}\pv\!\int_{\mathbb{T}}\!\!\frac{(1\!+\!f(\alpha))(1\!+\!f(\alpha\!-\!\beta))\sin{\left(\beta/2\right)}\!+\!(1\!+\!f(\alpha))\Delta_{\beta} f(\alpha)}{(\Delta_\beta f(\alpha))^2\!+\!(1\!+\!f(\alpha))(1\!+\!f(\alpha\!-\!\beta))}\frac{\tilde{\omega}(\alpha\!-\!\beta)}{2\sin{\left(\beta/2\right)}}d\beta\\
&\quad-\frac{1}{2\pi}\pv\int_{\mathbb{T}}\frac{(1+f(\alpha-\beta))\partial_{\alpha}f(\alpha)\cos{\left(\beta/2\right)}}{(\Delta_\beta f(\alpha))^2+(1+f(\alpha))(1+f(\alpha-\beta))}\frac{\tilde{\omega}(\alpha-\beta)}{2\sin{\left(\beta/2\right)}}d\beta.
\end{aligned}
\end{equation}
Plugging \eqref{z_param} into \eqref{eqcurve} we obtain that
\begin{equation}\label{fequation1aux}
\begin{aligned}
-R^2(1+f(\alpha))\partial_t f(\alpha)&=BR[z](\omega)(\alpha)\cdot\partial_\alpha z(\alpha)^\perp-\frac{A_\rho R}{A_{\rho\sigma}}\,\partial_\alpha\big((1+f(\alpha))\dot{c}\cdot \tau(\alpha)\big),
\end{aligned}
\end{equation}
and substituting the expression for $z(\alpha)$ \eqref{z_param} into the Birkhoff-Rott integral \eqref{eqBR} we conclude the evolution equation for $f(\alpha,t)$, 
\begin{equation*}
\begin{aligned}
\partial_t f(\alpha)&=\frac{A_\sigma}{R^3}N(f)(\alpha)\\
&=I_1(f)(\alpha)+I_2(f)(\alpha)+I_3(f)(\alpha),
\end{aligned}
\end{equation*}
where
\begin{equation*}
\begin{aligned}
I_1(f)(\alpha)&=\frac{1}{1+f(\alpha)}\partial_\alpha\big((1+f(\alpha))\dot{c}\cdot \tau(\alpha)\big),\\
I_2(f)(\alpha)&=-\frac{1/2\pi}{1+f(\alpha)}\pv\int_{\mathbb{T}}\!\frac{\partial_{\alpha}f(\alpha)\Delta_\beta f(\alpha)\!+\!(1\!+\!f(\alpha))(1\!+\!f(\alpha\!-\!\beta))\cos{\left(\beta/2\right)}}{(\Delta_\beta f(\alpha))^2+(1+f(\alpha))(1+f(\alpha-\beta))}\frac{\tilde{\omega}(\alpha\!-\!\beta)}{2\sin{\left(\beta/2\right)}}d\beta,
\\
I_3(f)(\alpha)&=-\frac{1/4\pi}{1+f(\alpha)}\int_{\mathbb{T}}\frac{\partial_{\alpha}f(\alpha)(1\!+\!f(\alpha\!-\!\beta))}{(\Delta_\beta f(\alpha))^2+(1+f(\alpha))(1+f(\alpha-\beta))}\tilde{\omega}(\alpha-\beta)d\beta.
\end{aligned}
\end{equation*}
For convenience and given that $A_\sigma>0$, we rescale time $t\rightarrow \frac{R^3}{A_\sigma}t$ so that 
\begin{equation*}
    \partial_t f(\alpha,\frac{R^3}{A_\sigma}t)=N(f)(\alpha,\frac{R^3}{A_\sigma}t).
\end{equation*}
Without loss of generality, we will abuse notation and keep $t$ as the rescale time variable:
\begin{equation}\label{fequation1}
\begin{aligned}
\partial_t f(\alpha)&=N(f)(\alpha)=I_1(f)(\alpha)+I_2(f)(\alpha)+I_3(f)(\alpha).
\end{aligned}
\end{equation}
Thus, our aim is to find $(f,\tilde{w})$ solving the equations \eqref{vorticity_tilde} and \eqref{fequation1} with initial data $f(\alpha,0)=f_0(\alpha)$ satisfying \eqref{f_cond}.

\subsection{Linearized operator}

We easily check that circles preserve their shape moving vertically with constant velocity. In fact, letting $f\equiv0$ in \eqref{fequation1} and \eqref{vorticity_tilde}, with $c(t)$ defined in \eqref{c_def}, we have

\begin{equation*}
\begin{aligned}
N_0(\alpha):=N(0)(\alpha)&=-A_{\rho\sigma}\sin{\alpha}-\frac{1}{2\pi}\pv\int_{\mathbb{T}} \frac{\cos{\left(\beta/2\right)}}{2\sin{\left(\beta/2\right)}}\tilde{\omega}_0(\alpha-\beta)d\beta,
\end{aligned}
\end{equation*}
and 
\begin{equation*}
\begin{aligned}
\tilde{\omega}_0(\alpha)&=\frac{A_\mu}{\pi}\int_{\mathbb{T}}\tilde{\omega}_0(\alpha-\beta)d\beta-2A_{\rho\sigma}\cos{\alpha}.
\end{aligned}
\end{equation*}
which thanks to \eqref{zeromean} reduces to
\begin{equation}\label{constantvel}
\begin{aligned}
    N_0(\alpha)&=-A_{\rho\sigma}\sin{\alpha}+A_{\rho\sigma} \sin{\alpha}=0,\\
    \tilde{\omega}_0(\alpha)&=-2A_{\rho\sigma}\cos{\alpha}.
\end{aligned}
\end{equation}
That is, for any value of the physical parameters, $f\equiv0$ (with corresponding $c_0(t)=(0,A_{\rho\sigma} t)$ \eqref{c_def}) gives a solution, which we call the trivial solution. We remark that, rescaling back the time, the dimensionless velocity $A_{\rho\sigma}$ corresponds to the true velocity $A_\rho$.
Next, we linearize the equations \eqref{vorticity_tilde} and \eqref{fequation1} around the trivial solution,
\begin{equation}\label{f_system}
\begin{aligned}
\partial_t f(\alpha)&=N_1(f)(\alpha)+N_{\geq2}(f)(\alpha),\\
\tilde{\omega}(\alpha)&=\tilde{\omega}_0(\alpha)+\tilde{\omega}_1(\alpha)+\tilde{\omega}_{\geq2}(\alpha),
\end{aligned}
\end{equation}
for which we compute the linear operators
\begin{equation}
    \begin{aligned}
        N_1(f)(\alpha)&=\frac{d}{d\varepsilon}N(\varepsilon f)(\alpha)\Big|_{\varepsilon=0},\\
        \tilde{\omega}_1(f)(\alpha)&=\frac{d}{d\varepsilon}\tilde{\omega}(\varepsilon f)(\alpha)\Big|_{\varepsilon=0}.
    \end{aligned}
\end{equation}
\begin{lemma}[Linearization] Let $f\in\mathcal{F}^{3,1}$ and $c(t)$ defined via \eqref{c_def}. The linear operators $N_1(f)$ and $\tilde{\omega}_1(f)$ are given by
    \begin{equation}\label{N_lin}
    \begin{aligned}        
        N_1(f)(\alpha)&=A_{\rho\sigma}(1-A_\mu)\Big(|\partial_\alpha|\big(f(\alpha)\sin(\alpha)\big)+\partial_\alpha\big(f(\alpha)\cos(\alpha)\big)\Big)\\
        &\quad-\big(|\partial_\alpha|^3f(\alpha)-|\partial_\alpha|f(\alpha)\big)-\dot{c}_{1}\cos(\alpha)-(\dot{c}_{2}-A_{\rho\sigma})\sin(\alpha),
 \end{aligned}
\end{equation}
    \begin{equation}\label{w_lin}
    \begin{aligned}        
        \tilde{\omega}_1(f)(\alpha)&=2A_\mu A_{\rho\sigma}\Big(\partial_\alpha\big(f(\alpha)\sin{(\alpha)}\big)-|\partial_\alpha|(f(\alpha)\cos(\alpha))\Big)-2\big(\partial_\alpha^3f(\alpha)+\partial_\alpha f(\alpha)\big)\\
&\quad-2A_{\rho\sigma}\partial_\alpha\big(f(\alpha)\sin(\alpha)\big).
    \end{aligned}
\end{equation}
\end{lemma}
\begin{proof}
    
First, we linearize the curvature \eqref{curvaturef}, 
\begin{equation}\label{Klin}
K_1(f)(\alpha)=\frac{d}{d\varepsilon}K(\varepsilon f)(\alpha)\Big|_{\varepsilon=0}=-\partial_\alpha^2f(\alpha)-f(\alpha),
\end{equation}
Next, we linearize $\mathcal{D}[f](\tilde{\omega})$ \eqref{Doper_f},
\begin{multline*}
    \mathcal{D}_1[f](\tilde{\omega})(\alpha)
    \\
    =
    \frac{1}{2\pi}\frac{d}{d\varepsilon}\Big(\pv\!\int_{\mathbb{T}}\!\!\frac{(1\!+\!f(\alpha))(1\!+\!f(\alpha\!-\!\beta))\sin{\left(\beta/2\right)}\!+\!(1\!+\!f(\alpha))\Delta_{\beta} f(\alpha)}{(\Delta_\beta f(\alpha))^2\!+\!(1\!+\!f(\alpha))(1\!+\!f(\alpha\!-\!\beta))}\frac{\tilde{\omega}(\alpha\!-\!\beta)}{2\sin{\left(\beta/2\right)}}d\beta\Big)\Big|_{\varepsilon=0}\\
\quad-\frac{1}{2\pi}\frac{d}{d\varepsilon}\Big(\pv\int_{\mathbb{T}}\frac{(1+f(\alpha-\beta))\partial_{\alpha}f(\alpha)\cos{\left(\beta/2\right)}}{(\Delta_\beta f(\alpha))^2+(1+f(\alpha))(1+f(\alpha-\beta))}\frac{\tilde{\omega}(\alpha-\beta)}{2\sin{\left(\beta/2\right)}}d\beta\Big)\Big|_{\varepsilon=0}.
\end{multline*}
We obtain that
\begin{equation*}
\begin{aligned}
\mathcal{D}_1[f](\tilde{\omega})(\alpha)&=\frac{1}{2\pi}\pv\int_{\mathbb{T}}\Delta_\beta f(\alpha)\frac{\tilde{\omega}_0(\alpha\!-\!\beta)}{2\sin{\left(\beta/2\right)}}d\beta+\frac{1}{2\pi}\pv\int_{\mathbb{T}}\frac{\tilde{\omega}_1(\alpha\!-\!\beta)}{2}d\beta\\
&\quad-\frac{1}{2\pi}\pv\int_{\mathbb{T}}\partial_{\alpha}f(\alpha)\cos{\left(\beta/2\right)}\frac{\tilde{\omega}_0(\alpha-\beta)}{2\sin{\left(\beta/2\right)}}d\beta,
\end{aligned}
\end{equation*}
and, since $\tilde{\omega}$ has mean zero, it simplifies
\begin{equation*}
\begin{aligned}
\mathcal{D}_1[f](\tilde{\omega})(\alpha)&=\frac{1}{2\pi}\pv\int_{\mathbb{T}}\Delta_\beta f(\alpha)\frac{\tilde{\omega}_0(\alpha\!-\!\beta)}{2\sin{\left(\beta/2\right)}}d\beta-\partial_{\alpha}f(\alpha)\mathcal{H}\tilde{\omega}_0(\alpha).
\end{aligned}
\end{equation*}
We substitute $\tilde{\omega}_0$ \eqref{constantvel} to get
\begin{equation}\label{Dlin}
\begin{aligned}
\mathcal{D}_1[f](\tilde{\omega})(\alpha)&=-\frac{A_{\rho\sigma}}{\pi}\pv\int_{\mathbb{T}}\frac{(f(\alpha)-f(\alpha-\beta))\cos{(\alpha-\beta)}}{4\sin^2{\left(\beta/2\right)}}d\beta+2A_{\rho\sigma}\partial_{\alpha}f(\alpha)\mathcal{H}\cos(\alpha)\\
&=A_{\rho\sigma}\Big(f(\alpha)\cos(\alpha)-|\partial_\alpha|(f(\alpha)\cos(\alpha))+\partial_{\alpha}f(\alpha)\sin(\alpha)\Big)\\
&=A_{\rho\sigma}\Big(\partial_\alpha\big(f(\alpha)\sin{(\alpha)}\big)-|\partial_\alpha|(f(\alpha)\cos(\alpha))\Big).
\end{aligned}
\end{equation}
Then, \eqref{Dlin} and \eqref{Klin} readily give the linearization \eqref{w_lin} of $\tilde{\omega}$ \eqref{vorticity_tilde}.
Finally, we compute $N_1(f)$. We linearize each $I_j(f)$, $j=1,2,3$, in the splitting for $N(f)$ \eqref{fequation1}:
\begin{equation*}
    \begin{aligned}
        N_1(f)(\alpha)&=\sum_{j=1}^3\frac{d}{d\varepsilon}I_j(\varepsilon f)(\alpha)\Big|_{\varepsilon=0}.
    \end{aligned}
\end{equation*}
Thanks to \eqref{zeromean}, the term $I_3(f)$ contains no linear terms. We find that
\begin{equation*}
    \begin{aligned}
        \frac{d}{d\varepsilon}I_1(\varepsilon f)(\alpha)\Big|_{\varepsilon=0}&=-\dot{c}_{1}\cos(\alpha)-(\dot{c}_{2}-A_{\rho\sigma}){A_\sigma}\sin(\alpha)+A_{\rho\sigma}\big(\partial_\alpha(\cos(\alpha)f(\alpha))+f(\alpha)\sin(\alpha)\big),
    \end{aligned}
\end{equation*}
and
\begin{equation*}
    \begin{aligned}
        \frac{d}{d\varepsilon}I_2(\varepsilon f)(\alpha)\Big|_{\varepsilon=0}&=-\frac{1}{2}\mathcal{H}(\tilde{\omega}_1)(\alpha)+\frac{1}{2}f(\alpha)\mathcal{H}(\tilde{\omega}_0)(\alpha).
    \end{aligned}
\end{equation*}
Therefore,
\begin{equation*}
    \begin{aligned}
        N_1(\alpha)&=-\dot{c}_{1}\cos(\alpha)-(\dot{c}_{2}-A_{\rho\sigma})\sin(\alpha)\!+\!A_{\rho\sigma}\cos(\alpha)\partial_\alpha f(\alpha)\!-\!\frac{1}{2}\mathcal{H}(\tilde{\omega}_1)(\alpha)+\frac{f(\alpha)}{2}\mathcal{H}(\tilde{\omega}_0)(\alpha).
    \end{aligned}
\end{equation*}
It follows from \eqref{constantvel} and \eqref{w_lin} that
\begin{equation*}
    \begin{aligned}
\mathcal{H}(\tilde{\omega}_0)(\alpha)&=-2A_{\rho\sigma}\sin(\alpha)\\
\mathcal{H}(\tilde{\omega}_1)(\alpha)&=2A_\mu A_{\rho\sigma}\Big(|\partial_\alpha|\big(f(\alpha)\sin(\alpha)\big)+\partial_\alpha\big(f(\alpha)\cos(\alpha)\big)\Big)+2\big(|\partial_\alpha|^3f(\alpha)-|\partial_\alpha|f(\alpha)\big)\\
&\quad-2A_{\rho\sigma}|\partial_\alpha|\big(f(\alpha)\sin(\alpha)\big),
    \end{aligned}
\end{equation*}
thus, introducing these expressions back into $N_1(\alpha)$,
\begin{equation*}
    \begin{aligned}
        N_1(\alpha)&=-\dot{c}_{1}\cos(\alpha)-(\dot{c}_{2}-A_{\rho\sigma})\sin(\alpha)+A_{\rho\sigma}\cos(\alpha)\partial_\alpha f(\alpha)\\
        &\quad-A_\mu A_{\rho\sigma}\Big(|\partial_\alpha|\big(f(\alpha)\sin(\alpha)\big)+\partial_\alpha\big(f(\alpha)\cos(\alpha)\big)\Big)-\big(|\partial_\alpha|^3f(\alpha)-|\partial_\alpha|f(\alpha)\big)\\
&\quad+A_{\rho\sigma}|\partial_\alpha|\big(f(\alpha)\sin(\alpha)\big)-A_{\rho\sigma}\sin(\alpha)f(\alpha),
    \end{aligned}
\end{equation*}
which concludes the result.

\end{proof}

We next decompose the linear system $\partial_t g(\alpha)=N_1(g)(\alpha)$ into frequencies. Let $g(\alpha)$ a real-valued function and let $g(\alpha)=\sum_{k\in\mathbb{Z}}\widehat{g}(k)e^{ik\alpha}$, with $\widehat{g}(k)$ the standard Fourier coefficient \eqref{Fourier_coef}. Then, substitution in \eqref{N_lin} and straightforward computations give that

\begin{equation*}
    \begin{aligned}
        \frac{d}{dt}\widehat{g}(k)&=-\frac12\delta_1(k)\big(\dot{c}_1-i(\dot{c}_2-A_{\rho\sigma})\big)-\frac12\delta_{-1}(k)\big(\dot{c}_1+i(\dot{c}_2-A_{\rho\sigma})\big)\\
        &\quad-|k|\big(|k|^2-1\big)\widehat{g}(k)\\
        &\quad-i(1-A_\mu)A_{\rho\sigma}\Big(\frac{|k|-k}{2}\widehat{g}(k-1)-\frac{|k|+k}{2}\widehat{g}(k+1)\Big).
    \end{aligned}
\end{equation*}
Since $g(\alpha)$ is real, $\widehat{g}(-k)=\overline{\widehat{g}(k)}$ and we only need to study $k\geq0$:
\begin{equation*}
    \begin{aligned}
        \frac{d}{dt} \widehat{g}(k)&=-k\big(k^2-1\big)\widehat{g}(k)+i(1-A_\mu)A_{\rho\sigma}k\widehat{g}(k+1)\\
        &\quad-\frac{\delta_{1}(k)}{2}\big(\dot{c}_1-i(\dot{c}_2-A_{\rho\sigma})\big).
    \end{aligned}
\end{equation*}
Plugging in the specific expression of $c(t)$ in \eqref{c_def} we obtain
\begin{equation*}
    \begin{aligned}
        \frac{d}{dt} \widehat{g}(k)&=-k\big(k^2-1+\delta_1(k)\big)\widehat{g}(k)+i(1-A_\mu)A_{\rho\sigma}k\widehat{g}(k+1).
    \end{aligned}
\end{equation*}
For $k\geq1$, this is an upper triangular system with eigenvalues $-k(k^2-1+\delta_1(k))$. 
We write the system as follows
\begin{equation}\label{linearZ}
    \begin{aligned}
    \dot{z}_k=\sum_{j\geq1}M_{k,j}z_k,\quad k\geq1,
    \end{aligned}
\end{equation}
where we denote
\begin{equation*}
    M_{k,j}=\left\{
    \begin{aligned}
    &-a_k,\quad j=k,\\
    &b_k,\hspace{0.8cm} j=k+1,\\
    &0,\hspace{1.2cm} \text{otherwise},
    \end{aligned}\right.
\end{equation*}
with 
\begin{equation}\label{coeffs}
    a_k= k(k^2-1+\delta_1(k)),\quad b_k=i(1-A_\mu)A_{\rho\sigma}k,
\end{equation}
The diagonalization of this system is done in Proposition 7.1, Lemma 7.2, and Remark 7.3 of \cite{GancedoGarciaJuarezPatelStrain23}, whose statements we include here for completeness.
\begin{prop}[Diagonalization]\label{diagonalization}
Consider $z=(z_k)_{k\geq1}$, $y=(y_k)_{k\geq1}\in \ell^1$ and the linear operator \begin{equation*}
    \begin{aligned}
    S^{-1}:\ell^1&\rightarrow \ell^1\\
    z&\mapsto y=S^{-1}z,
    \end{aligned}
\end{equation*}
defined by
    \begin{equation*}
        y_k=\sum_{j\geq1}S^{-1}_{k,j}z_j,
    \end{equation*}
    with
    \begin{equation}\label{Sm1}
        S^{-1}_{k,j}=
        \begin{cases}
        (-1)^{j-k}\prod_{l=1}^{j-k} \frac{b_{k-1+l}}{a_k-a_{k+l}},& j\geq k\geq2,\\
        1,&j=k=1,\\
        0,& \text{otherwise}.
        \end{cases}
    \end{equation}
Then, the inverse operator $S$ is given by
    \begin{equation*}
        S_{k,j}=
        \begin{cases}
        \prod_{l=1}^{j-k} \frac{b_{k-1+l}}{a_{k-1+l}-a_j},&j\geq k\geq2,\\
        1,&j=k=1,\\
        0,& \text{otherwise}.
        \end{cases}
    \end{equation*}
Moreover, the linear operator $S^{-1}$ diagonalizes the system \eqref{linearZ} as follows
\begin{equation*}
    \dot{y}_k=-a_ky_k,\quad k\geq1.
\end{equation*}
\end{prop}

We remark that since $\prod_{l=1}^{0}:=1$ by definition then $S_{k,k} = S^{-1}_{k,k} =1$ when $j=k$ for all $k\ge 1$.   We also have the following lemma which gives uniform bounds for the operators $S$ and $S^{-1}$.

\begin{lemma}\label{lemmaS}
    The operator norms in $\ell^1$ of the linear operators $S$ and $S^{-1}$ satisfy the following bounds
    \begin{equation}\label{CSbounds}
        \begin{aligned}
        \|S^{-1}\|_{\ell^1\rightarrow \ell^1}\leq C_S,\qquad  \|S\|_{\ell^1\rightarrow \ell^1}\leq C_S,
        \end{aligned}
    \end{equation}
with $C_S = C_S\big(|A_{\rho\sigma}|(1-A_\mu)\big)$.
\end{lemma}
\begin{remark}
The results in Proposition \ref{diagonalization} and Lemma \ref{lemmaS} also hold in the space $\ell^1$ with weight $e^{\nu(t)|k|}|k|^s$ for any $s\ge 0$:
\begin{equation*}
    \|z\|_{\fsonenu}
    \leq 
     C_S\|y\|_{\fsonenu},
     \quad 
         \|y\|_{\fsonenu}
    \leq 
     C_S\|z\|_{\fsonenu}.
\end{equation*}
\end{remark}

\subsection{Fixed point formulation}

Let us denote $\mathbb{P}:\mathcal{F}^{0,1}(\mathbb{T})\rightarrow Z$ the projection operator onto the space of functions with mean zero:
\begin{equation*}
    \begin{aligned}
        \mathbb{P}f(\alpha)=f(\alpha)-\widehat{f}(0).
    \end{aligned}
\end{equation*}
We then decompose the evolution system \eqref{fequation1}:
\begin{equation}\label{system1}
    \begin{aligned}
        \partial_t \widehat{f}(0)&=\widehat{N(f)}(0),\\
        \partial_t\mathbb{P}f(\alpha)&=\mathbb{P}N(f)(\alpha).
    \end{aligned}
\end{equation}
Next lemma shows that we can replace the first equation due to the fact that the fluids are incompressible. The conservation of the bubble area implies that, for all $t\geq0$,

\begin{equation*}
    \begin{aligned}
        \int_{\mathbb{T}}f(\alpha)d\alpha=-\frac{1}{2}\int_{\mathbb{T}}(f(\alpha))^2d\alpha.
    \end{aligned}
\end{equation*}
\begin{lemma} 
Let $f\in C((0,T];\mathcal{F}^{3,1})\cap C^1((0,T];\mathcal{F}^{0,1})$, $\omega\in C((0,T];Z)$ and $\|f\|_{\mathcal{F}^{0,1}}<1$. Then, $(f, \omega)$ is a solution to the equations \eqref{fequation1} and \eqref{vorticity_tilde}  if and only if it is a solution to 
\begin{equation}\label{system_used}
    \left\{\begin{aligned}
        \widehat{f}(0)&=-\frac{1}{4\pi}\int_{\mathbb{T}}(f(\alpha))^2d\alpha,\\
        \partial_t\mathbb{P}f(\alpha)&=\mathbb{P}N(f)(\alpha),\\
\tilde{\omega}(\alpha)&=2 A_\mu \mathcal{D}[f](\tilde{\omega})(\alpha)+2\partial_{\alpha}K(f)(\alpha)-2A_{\rho\sigma} \partial_\alpha\big((1+f(\alpha))\sin{\alpha}\big).
    \end{aligned}\right.
\end{equation}
\end{lemma}

\begin{proof}
We first show that \eqref{vorticity_tilde}-\eqref{fequation1} imply \eqref{system_used}. First, we notice the following identity
\begin{equation*}
    \begin{aligned}
        \int_{\mathbb{T}}BR[z](\tilde{\omega})(\alpha)\cdot\partial_\alpha z(\alpha)^\perp d\alpha=0,
    \end{aligned}
\end{equation*}
which follows from the definition of the Birkhoff-Rott integral and the fact that $\tilde{\omega}\in Z$,
\begin{equation}\label{aux2}
    \begin{aligned}
        BR[z](\tilde{\omega})(\alpha)\cdot\partial_\alpha z(\alpha)^\perp&=\frac1{2\pi} \pv\int_{\mathbb{T}} \frac{(z(\alpha)-z(\beta))\cdot \partial_\alpha z(\alpha)}{|z(\alpha)-z(\beta)|^2}\tilde{\omega}(\beta)d\beta\\
        &=\frac1{2\pi}\partial_\alpha\Big(\pv\int_{\mathbb{T}}  \log{\big(|z(\alpha)-z(\beta)|\big)}\tilde{\omega}(\beta)d\beta\Big).
    \end{aligned}
\end{equation}
Then, integration in \eqref{fequation1aux} gives
\begin{equation*}
    \begin{aligned}
        \int_{\mathbb{T}}(1+f(\alpha))\partial_t f(\alpha)d\alpha=0, 
    \end{aligned}
\end{equation*}
thus, for $t\geq0$,
\begin{equation*}
    \begin{aligned}
        \int_{\mathbb{T}}f(\alpha)d\alpha=-\frac12\int_{\mathbb{T}}(f(\alpha))^2d\alpha.
    \end{aligned}
\end{equation*}
Next, we show that \eqref{system_used} imply \eqref{system1}. We take a time derivative to the first equation of \eqref{system_used},
\begin{equation*}
    \begin{aligned}
        \partial_t\widehat{f}(0)&=-\frac{1}{2\pi}\int_{\mathbb{T}}f(\alpha)\partial_t f(\alpha)d\alpha\\
        &=-\frac{1}{2\pi}\int_{\mathbb{T}}f(\alpha)\partial_t\big( \mathbb{P}f(\alpha)+\widehat{f}(0)\big)d\alpha,
    \end{aligned}
\end{equation*}
and then insert the second equation of \eqref{system_used},
\begin{equation*}
    \begin{aligned}
        \partial_t\widehat{f}(0)&=-\frac{1}{2\pi}\int_{\mathbb{T}}f(\alpha)\mathbb{P} N(f)(\alpha)d\alpha-\frac{1}{2\pi}\int_{\mathbb{T}}f(\alpha)\partial_t\widehat{f}(0)d\alpha\\
        &=-\frac{1}{2\pi}\int_{\mathbb{T}}\mathbb{P}f(\alpha) N(f)(\alpha)d\alpha-\widehat{f}(0)\partial_t\widehat{f}(0)\\
        &=-\frac{1}{2\pi}\int_{\mathbb{T}}f(\alpha) N(f)(\alpha)d\alpha+\frac{1}{2\pi}\int_{\mathbb{T}}\widehat{f}(0) N(f)(\alpha)d\alpha-\widehat{f}(0)\partial_t\widehat{f}(0).
    \end{aligned}
\end{equation*}
We can write the previous identity \eqref{aux2} in terms of $f$:
\begin{equation*}
    \begin{aligned}
        \int_{\mathbb{T}}\big(1+f(\alpha)\big)N(f)(\alpha)d\alpha=0,
    \end{aligned}
\end{equation*}
thus we get
\begin{equation*}
    \begin{aligned}
        \partial_t\widehat{f}(0)&=-\frac{1}{2\pi}\int_{\mathbb{T}}\big(1+f(\alpha)\big) N(f)(\alpha)d\alpha+\big(1+\widehat{f}(0)\big)\frac{1}{2\pi}\int_{\mathbb{T}} N(f)(\alpha)d\alpha-\widehat{f}(0)\partial_t\widehat{f}(0).
    \end{aligned}
\end{equation*}
Hence, since $\|f\|_{\mathcal{F}^{0,1}}<1$,
\begin{equation*}
    \begin{aligned}
        \partial_t\widehat{f}(0)&=\frac{1}{2\pi}\int_{\mathbb{T}}N(f)(\alpha)d\alpha=\widehat{N(f)}(0).
    \end{aligned}
\end{equation*}
   
\end{proof}

The previous lemma and Parseval's identity shows that at any time instant, the zero frequency is determined by $\mathbb{P}f$:
\begin{equation}\label{zero_freq}
    \begin{aligned}
        \widehat{f}(0)&=-1+\sqrt{1-\sum_{k\in\mathbb{Z}\setminus\{0\}}|\widehat{f}(k)|^2}=-1+\sqrt{1-\frac{1}{2\pi}\|\mathbb{P}f\|_{L^2}^2}.
    \end{aligned}
\end{equation}
Let us denote $\mathcal{S}$ and $\mathcal{S}^{-1}$ the linear operators in physical space associated to $S$ and $S^{-1}$ defined in Proposition \ref{diagonalization}, and $\mathcal{A}$ the linear, diagonal, positive operator defined via the Fourier multiplier $a(k)$: $\mathcal{F}\big(\mathcal{A}f\big)(k):=a(k)\widehat{f}(k)$.
We further recall the notation 
\begin{equation*}
N_{\geq2}(f)=N(f)(\alpha)-N_1(f)(\alpha),    
\end{equation*}
so that equation \eqref{f_system} writes as
\begin{equation}\label{N_system2}
\begin{aligned}
\partial_t f(\alpha)&=N_1(f)(\alpha)+N_{\geq2}(f)(\alpha).
\end{aligned}
\end{equation}
Applying $\mathbb{P}$ and $\mathcal{S}^{-1}$ the equation rewrites as 
\begin{equation*}
\begin{aligned}
\partial_t \mathcal{S}^{-1}\mathbb{P}f(\alpha)&=\mathcal{A}\mathcal{S}^{-1}\mathbb{P}f(\alpha)+\mathcal{S}^{-1}\mathbb{P}N_{\geq2}(f)(\alpha).
\end{aligned}
\end{equation*}
 We can write the equation into its mild formulation using Duhamel formula:
\begin{equation}\label{duhamel_eq}
    \begin{aligned}
        \mathbb{P}f(\alpha)=\mathcal{S}e^{-\mathcal{A}t}\mathcal{S}^{-1}\mathbb{P}f_0(\alpha)+\int_0^t\big(\mathcal{S}e^{-(t-\tau)\mathcal{A}}\mathcal{S}^{-1}\mathbb{P}N_{\geq2}(f)(\tau)\big)(\alpha)d\tau.
    \end{aligned}
\end{equation}
Finally, we formulate our problem \eqref{system_used} as a fixed point equation in the space $X_{T,\nu}(\mathbb{T})$, given in \eqref{Xspace}:
\begin{equation*}
    \begin{aligned}
        \mathbb{P}f&=\mathcal{T}(\mathbb{P}f),
    \end{aligned}
\end{equation*}
with $\mathcal{T}:X_{T,\nu}(\mathbb{T})\longrightarrow X_{T,\nu}(\mathbb{T})$ defined by
\begin{equation*}
    \begin{aligned}
        \mathcal{T}(\mathbb{P}f)&:=\mathcal{S}e^{-\mathcal{A}t}\mathcal{S}^{-1}\mathbb{P}f_0+\int_0^t\mathcal{S}e^{-(t-\tau)\mathcal{A}}\mathcal{S}^{-1}\mathbb{P}N_{\geq2}(f)(\tau)d\tau.
    \end{aligned}
\end{equation*}
We notice that $N_{\geq2}(f)$ above only depends on $\mathbb{P}f$ since, for each fixed time, we will show that the $\omega$ equation in \eqref{system_used} is uniquely solvable and gives $\omega=\omega(f)$, and by \eqref{zero_freq}, we can write
\begin{equation*}
    f=\widehat{f}(0)+\mathbb{P}f=\mathbb{P}f-1+\sqrt{1-\frac{1}{2\pi}\|\mathbb{P}f\|_{L^2}^2}.
\end{equation*}
Therefore, it holds that for any $s\geq0$,
\begin{equation}\label{zerof11}
    \begin{aligned}
    |\widehat{f}(0)|&\lesssim_1 \|f\|_{\dot{\mathcal{F}}^{0,1}},\\
    \|f\|_{\mathcal{F}^{s,1}_\nu}&=\widehat{f}(0)+ \|f\|_{\dot{\mathcal{F}}^{s,1}_\nu}\lesssim_1 \|f\|_{\dot{\mathcal{F}}^{s,1}_\nu}. 
    \end{aligned}
\end{equation}
With abuse of notation, we have the fixed point equation
\begin{equation}\label{fixed_p}
    \begin{aligned}
     \mathbb{P}f&=\mathcal{T}(\mathbb{P}f),\\
        \mathcal{T}(\mathbb{P}f)&:=\mathcal{S}e^{-\mathcal{A}t}\mathcal{S}^{-1}\mathbb{P}f_0+\int_0^t\mathcal{S}e^{-(t-\tau)\mathcal{A}}\mathcal{S}^{-1}\mathbb{P}\tilde{N}_{\geq2}(\mathbb{P}f)(\tau)d\tau,
    \end{aligned}
\end{equation}
where $\tilde{N}_{\geq2}(\mathbb{P}f):=N_{\geq2}\big(\mathbb{P}f-1+\sqrt{1-\frac{1}{2\pi}\|\mathbb{P}f\|_{L^2}^2}\big)$.
\begin{thm}\label{fixed_theo}
     Let $f_0\in \mathcal{F}^{1,1}(\mathbb{T})$ and $A_\sigma>0$. For any $A_\mu\in[-1,1]$, $A_{\rho\sigma}\in\mathbb{R}$, there exists $\nu>0$ and $\varepsilon=\varepsilon(A_\mu, |A_{\rho\sigma}|)>0$  such that, for any $T>0$, if $\|f_0\|_{\mathcal{F}^{1,1}}\leq \varepsilon$, then  $\mathcal{T}$ forms a contraction in the ball $B_{\varepsilon,T,\nu}(u_0)\subset X_{T,\nu}(\mathbb{T})$, with
     \begin{equation*}
         u_0:=\mathcal{S}e^{-\mathcal{A}t}\mathcal{S}^{-1}\mathbb{P}f_0.
     \end{equation*}
\end{thm}
\begin{proof} 
First, Lemmas \ref{lemmaS} and \ref{semig_lem} show that $u_0$ is  contained in a ball of $X_{T,\nu}$ centered at the origin and with radius proportional to $\varepsilon$: 
\begin{equation*}
    \begin{aligned}
        \|u_0\|_{X_{T,\nu}}&\lesssim C_S\|f_0\|_{\dot{\mathcal{F}}^{1,1}},
    \end{aligned}
\end{equation*}
with $C_S = C_S\big(|A_{\rho\sigma}|(1-A_\mu)\big)$. The continuity in time at $t=0^+$ in $Z^{1,1}_\nu$  follows from having $f\in L^1(0,T;\dot{\mathcal{F}}^{4,1}_\nu)$ and Duhamel's formula.
Therefore, if $f\in B_{\varepsilon, T,\nu}(u_0)$, it is also contained in a ball of radius comparable to $\varepsilon$:
\begin{equation*}
    \begin{aligned}
        \|f\|_{X_{T,\nu}}&\lesssim(1+C_S)\varepsilon.
    \end{aligned}
\end{equation*}
We next show that $\mathcal{T}$ maps the ball $B_{\varepsilon,T,\nu}(u_0)$ into itself. Let $f\in B_{\varepsilon,T,\nu}(u_0)$. We have
\begin{equation*}
    \begin{aligned}
        \|\mathcal{T}(f)-u_0\|_{X_{T},\nu}&\leq \Big|\Big|\int_0^t\mathcal{S}e^{-(t-\tau)\mathcal{A}}\mathcal{S}^{-1}\mathbb{P}\tilde{N}_{\geq2}(f)(\tau)d\tau\Big|\Big|_{X_{T},\nu},
    \end{aligned}
\end{equation*}
which by Lemmas \ref{lemmaS} and \ref{semig_lem} is bounded as follows
\begin{equation*}
    \begin{aligned}
        \|\mathcal{T}(f)-u_0\|_{X_{T},\nu}&\lesssim C_S\int_0^T\|\tilde{N}_{\geq2}(f)(\tau)\|_{\dot{\mathcal{F}}^{1,1}_\nu}d\tau.
    \end{aligned}
\end{equation*}
The expression of $N_{\geq2}(f)$ as the nonlinear in $f$ part of $N(f)$ \eqref{fequation1}, together with Lemma \ref{nonlinearlemma}, yields that
\begin{equation}\label{auxN}
    \begin{aligned}
        \|\tilde{N}_{\geq2}(f)\|_{\dot{\mathcal{F}}^{1,1}_\nu}&\lesssim_1 \big(\|f\|_{\dot{\mathcal{F}}^{2,1}_\nu}\|f\|_{\dot{\mathcal{F}}^{1,1}_\nu}+\|f\|_{\dot{\mathcal{F}}^{1,1}_\nu}^2\big)+|A_{\rho\sigma}|\big(\|f\|_{\dot{\mathcal{F}}^{2,1}_\nu}\|f\|_{\dot{\mathcal{F}}^{1,1}_\nu}+\|f\|_{\dot{\mathcal{F}}^{1,1}_\nu}^2\big)\\
        &\quad+\|\tilde{\omega}_{\geq1}\|_{\dot{\mathcal{F}}^{0,1}_\nu}\|f\|_{\dot{\mathcal{F}}^{2,1}_\nu}+\|\tilde{\omega}_{\geq1}\|_{\dot{\mathcal{F}}^{1,1}_\nu}\|f\|_{\dot{\mathcal{F}}^{1,1}_\nu},
    \end{aligned}
\end{equation}
where we have repeatedly used \eqref{zerof11} and \eqref{Wienerorder}. Moreover, since $f\in Z$, $\|f\|_{\mathcal{F}^{s,1}_\nu}=\|f\|_{\dot{\mathcal{F}}^{s,1}_\nu}$ for any $s\geq0$. We plug in the vorticity estimates from Corollary \ref{vorticity_cor} to obtain
\begin{equation*}
    \begin{aligned}
        \|\tilde{N}_{\geq2}(f)\|_{\dot{\mathcal{F}}^{1,1}_\nu}&\lesssim_1 \|f\|_{\dot{\mathcal{F}}^{2,1}_\nu}\|f\|_{\dot{\mathcal{F}}^{1,1}_\nu}+\|f\|_{\dot{\mathcal{F}}^{3,1}_\nu}\|f\|_{\dot{\mathcal{F}}^{2,1}_\nu}+\|f\|_{\dot{\mathcal{F}}^{4,1}_\nu}\|f\|_{\dot{\mathcal{F}}^{1,1}_\nu}\\
        &\quad+|A_{\rho\sigma}|(1+A_\mu)\|f\|_{\dot{\mathcal{F}}^{2,1}_\nu}\|f\|_{\dot{\mathcal{F}}^{1,1}_\nu},
    \end{aligned}
\end{equation*}
thus interpolation (see Lemma \ref{interp_lem})  followed by time integration gives
\begin{equation*}
    \begin{aligned}
        \int_0^T\!\!\|\tilde{N}_{\geq2}(f)(\tau)\|_{\dot{\mathcal{F}}^{1,1}_\nu}d\tau&\lesssim_1 \!\big(1\!+\!|A_{\rho\sigma}|(1\!+\!A_\mu)\big)\int_0^T\!\!\|f\|_{\dot{\mathcal{F}}^{4,1}_\nu}\|f\|_{\dot{\mathcal{F}}^{1,1}_\nu}d\tau\lesssim_1 \big(1\!+\!|A_{\rho\sigma}|(1\!+\!A_\mu)\big)\varepsilon^2.
    \end{aligned}
\end{equation*}
Finally, we show that $\mathcal{T}$ is a contraction in $B_{\varepsilon, T,\nu}(u_0)$. Let $f_1, f_2\in B_{\varepsilon,T,\nu}$. Then, using Lemmas \ref{lemmaS} and \ref{semig_lem} as before,
\begin{equation}\label{contract_est}
    \begin{aligned}
        \|\mathcal{T}(f_1)-\mathcal{T}(f_2)\|_{X_{T},\nu}&\lesssim C_S\int_0^T\big(\|\tilde{N}_{\geq2}(f_1)(\tau)-\tilde{N}_{\geq2}(f_2)(\tau)\|_{\dot{\mathcal{F}}^{1,1}_\nu}\big)d\tau.
    \end{aligned}
\end{equation}
We perform the difference of nonlinearities and apply Lemma \ref{nonlinearlemma} to obtain, following the steps done in  \eqref{auxN},
\begin{equation*}
    \begin{aligned}
        \|\tilde{N}_{\geq2}(f_1)-\tilde{N}_{\geq2}(f_2)\|_{\dot{\mathcal{F}}^{1,1}_\nu}\lesssim_1 L_1+L_2+L_3,
    \end{aligned}
\end{equation*}
where
\begin{equation*}
    \begin{aligned}
      L_1&=\|f_1-f_2\|_{\dot{\mathcal{F}}^{2,1}_\nu}\big(\|f_1\|_{\dot{\mathcal{F}}^{1,1}_\nu}+\|f_2\|_{\dot{\mathcal{F}}^{1,1}_\nu}\big)+\|f_1-f_2\|_{\dot{\mathcal{F}}^{1,1}_\nu}\big(\|f_1\|_{\dot{\mathcal{F}}^{2,1}_\nu}+\|f_2\|_{\dot{\mathcal{F}}^{2,1}_\nu}\big),
    \end{aligned}
\end{equation*}
\begin{equation*}
    \begin{aligned}
      L_2&=|A_{\rho\sigma}|\Big(\|f_1-f_2\|_{\dot{\mathcal{F}}^{2,1}_\nu}\big(\|f_1\|_{\dot{\mathcal{F}}^{1,1}_\nu}+\|f_2\|_{\dot{\mathcal{F}}^{1,1}_\nu}\big)+\|f_1-f_2\|_{\dot{\mathcal{F}}^{1,1}_\nu}\big(\|f_1\|_{\dot{\mathcal{F}}^{2,1}_\nu}+\|f_2\|_{\dot{\mathcal{F}}^{2,1}_\nu}\big)\Big),
    \end{aligned}
\end{equation*}
and
\begin{equation*}
    \begin{aligned}
       L_3&=\|\tilde{\omega}_{\geq1}(f_1)-\tilde{\omega}_{\geq1}(f_2)\|_{\dot{\mathcal{F}}^{0,1}_\nu}\big(\|f_1\|_{\dot{\mathcal{F}}^{2,1}_\nu}+\|f_2\|_{\dot{\mathcal{F}}^{2,1}_\nu}\big)\\
       &\quad+\big(\|\tilde{\omega}_{\geq1}(f_1)\|_{\dot{\mathcal{F}}^{0,1}_\nu}+\|\tilde{\omega}_{\geq1}(f_2)\|_{\dot{\mathcal{F}}^{0,1}_\nu}\big)\|f_1-f_2\|_{\dot{\mathcal{F}}^{2,1}_\nu}\\
        &\quad+\|\tilde{\omega}_{\geq1}(f_1)-\tilde{\omega}_{\geq1}(f_2)\|_{\dot{\mathcal{F}}^{1,1}_\nu}\big(\|f_1\|_{\dot{\mathcal{F}}^{1,1}_\nu}+\|f_2\|_{\dot{\mathcal{F}}^{1,1}_\nu}\big)\\
        &\quad+\big(\|\tilde{\omega}_{\geq1}(f_1)\|_{\dot{\mathcal{F}}^{1,1}_\nu}+\|\tilde{\omega}_{\geq1}(f_2)\|_{\dot{\mathcal{F}}^{1,1}_\nu}\big)\|f_1-f_2\|_{\dot{\mathcal{F}}^{1,1}_\nu}.
    \end{aligned}
\end{equation*}
We immediately find that
we find that 
\begin{equation*}
    \begin{aligned}
      \int_0^T \big(L_1(\tau)+L_2(\tau)\big)d\tau&\lesssim_1 \varepsilon\big(1\!+\!|A_{\rho\sigma}|\big)\Big(\!\int_0^T\!\!\|f_1\!-\!f_2\|_{\dot{\mathcal{F}}^{2,1}_\nu}d\tau\!+\!\sup_{0\leq \tau\leq t}\|f_1\!-\!f_2\|_{\dot{\mathcal{F}}^{1,1}_\nu}\!\Big)\\
      &=\varepsilon\big(1\!+\!|A_{\rho\sigma}|\big) \|f_1-f_2\|_{X_{T,\nu}}.
    \end{aligned}
\end{equation*}
As done in Corollary \ref{vorticity_cor}, the vorticity satisfies the following contraction estimates:
\begin{equation*}
    \begin{aligned}
    \|\tilde{w}_{\geq1}(f_1)-\tilde{w}_{\geq1}(f_2)\|_{\dot{\mathcal{F}}^{0,1}_\nu}&\lesssim_1 \|f_1-f_2\|_{\dot{\mathcal{F}}^{3,1}_\nu}+|A_{\rho\sigma}|(1+A_\mu)\|f_1-f_2\|_{\dot{\mathcal{F}}^{1,1}_\nu},\\
    \|\tilde{w}_{\geq1}(f_1)-\tilde{w}_{\geq1}(f_2)\|_{\dot{\mathcal{F}}^{1,1}_\nu}&\lesssim_1\|f_1-f_2\|_{\dot{\mathcal{F}}^{4,1}_\nu}+|A_{\rho\sigma}|(1+A_\mu)\|f_1-f_2\|_{\dot{\mathcal{F}}^{2,1}_\nu},
    \end{aligned}
\end{equation*}
where here the notation $\lesssim_1$ includes dependency on $\|f_1\|_{\dot{\mathcal{F}}^{1,1}_\nu}$ and $\|f_2\|_{\dot{\mathcal{F}}^{1,1}_\nu}$.
Interpolation followed by Young's inequality gives that 
\begin{equation*}
    \begin{aligned}
       \int_0^T \|f_1-f_2\|_{\dot{\mathcal{F}}^{3,1}_\nu}\|f_i\|_{\dot{\mathcal{F}}^{2,1}_\nu}d\tau&\lesssim  \int_0^T \|f_1-f_2\|_{\dot{\mathcal{F}}^{1,1}_\nu}\|f_i\|_{\dot{\mathcal{F}}^{4,1}_\nu}d\tau+ \int_0^T \|f_1-f_2\|_{\dot{\mathcal{F}}^{4,1}_\nu}\|f_i\|_{\dot{\mathcal{F}}^{1,1}_\nu}d\tau\\
       &\lesssim \varepsilon\Big( \sup_{0\leq\tau\leq T} \|f_1-f_2\|_{\dot{\mathcal{F}}^{1,1}_\nu}d\tau+ \int_0^T \|f_1-f_2\|_{\dot{\mathcal{F}}^{4,1}_\nu}d\tau\Big)\\
       &=\varepsilon \|f_1-f_2\|_{X_{T,\nu}},
    \end{aligned}
\end{equation*}
and analogously
\begin{equation*}
    \begin{aligned}
       \int_0^T \|f_1-f_2\|_{\dot{\mathcal{F}}^{2,1}_\nu}\|f_i\|_{\dot{\mathcal{F}}^{3,1}_\nu}d\tau&\lesssim\varepsilon \|f_1-f_2\|_{X_{T,\nu}}.
    \end{aligned}
\end{equation*}
Hence we obtain
\begin{equation*}
    \begin{aligned}
      \int_0^T L_3(\tau)d\tau&\lesssim_1 \varepsilon \big(1+|A_{\rho\sigma}|(1+A_\mu)\big)\|f_1-f_2\|_{X_{T,\nu}},
    \end{aligned}
\end{equation*}
and substituting back in \eqref{contract_est} we conclude that
\begin{equation*}
    \begin{aligned}
        \|\mathcal{T}(f_1)-\mathcal{T}(f_2)\|_{X_{T},\nu}&\lesssim_1 C_S
        \big(1+|A_{\rho\sigma}|(1-A_\mu)\big)\varepsilon\|f_1-f_2\|_{X_{T,\nu}},
    \end{aligned}
\end{equation*}
thus choosing $\varepsilon=\varepsilon(A_\mu, |A_{\rho\sigma}|)>0$ sufficiently small we conclude that $\mathcal{T}$ is a contraction.

\end{proof}

\begin{lemma}[Semigroup estimate]\label{semig_lem}
   Let $g_0\in\dot{\mathcal{F}}^{1,1}$ and $g\in L^\infty(0,T;\dot{\mathcal{F}}^{1,1}_\nu)\cap L^1(0,T;\dot{\mathcal{F}}^{4,1}_\nu)$  be real-valued functions, and denote $F(t,\alpha)=\int_0^t e^{-(t-\tau)\mathcal{A}}g(\tau,\cdot)(\alpha)d\tau.$
   Then, for $\nu\geq0$ small enough,
   \begin{equation}\label{semig1}
   \begin{aligned}
               \sup_{t\geq0}\|e^{-t\mathcal{A}}g_0\|_{\dot{\mathcal{F}}^{1,1}_\nu}&\leq \|g_0\|_{\dot{\mathcal{F}}^{1,1}},\\
        \|F(t)\|_{\dot{\mathcal{F}}^{1,1}_\nu}&\leq\int_0^t \|g(\tau)\|_{\dot{\mathcal{F}}^{1,1}_\nu}d\tau,      
   \end{aligned}
   \end{equation}
  and
  \begin{equation}\label{semig2}
        \begin{aligned}
        \int_0^t\|e^{-\tau\mathcal{A}}g_0\|_{\dot{\mathcal{F}}^{4,1}_\nu}d\tau&\leq C(\nu) \|g_0\|_{\dot{\mathcal{F}}^{1,1}}\lesssim \|g_0\|_{\dot{\mathcal{F}}^{1,1}},\\
\int_0^t\|F(\tau)\|_{\dot{\mathcal{F}}^{4,1}_\nu}d\tau&\leq C(\nu)\int_0^t \|g(\tau)\|_{\dot{\mathcal{F}}^{1,1}_\nu}d\tau\lesssim\int_0^t \|g(\tau)\|_{\dot{\mathcal{F}}^{1,1}_\nu}d\tau.
        \end{aligned}
    \end{equation}
\end{lemma}
\begin{proof}
By definition of the operator $\mathcal{A}$ and the Wiener norm, we get
\begin{equation*}
    \begin{aligned}
        \|e^{-t\mathcal{A}}g_0\|_{\dot{\mathcal{F}}_{\nu}^{1,1}}&=2\sum_{k\geq1}|k| e^{\nu|k|t}e^{-|a_k| t}|\widehat{g}_0(k)|\leq \|g_0\|_{\dot{\mathcal{F}}^{1,1}}.
    \end{aligned}
\end{equation*}
Analogously, 
\begin{equation*}
    \begin{aligned}
        \|F(t)\|_{\dot{\mathcal{F}}_{\nu}^{1,1}}&\leq2\sum_{k\geq1}|k|e^{\nu|k|t}\int_0^t e^{-|a_k| (t-\tau)}|\widehat{g}(k)|d\tau\\
  &\leq2\sum_{k\geq1}|k|\int_0^t e^{\nu|k|(t-\tau)}e^{-|a_k| (t-\tau)}e^{\nu|k|\tau}|\widehat{g}(k)|d\tau
                \leq \int_0^t \|g(\tau)\|_{\dot{\mathcal{F}}^{1,1}_\nu}d\tau.
    \end{aligned}
\end{equation*}
Next,
\begin{equation*}
    \begin{aligned}
        \int_0^t\|e^{-\tau\mathcal{A}}g_0\|_{\dot{\mathcal{F}}_{\nu}^{4,1}}d\tau&=\int_0^t2\sum_{k\geq1}|k|^{3} e^{-|a_k| \tau}e^{\nu|k|\tau}|k||\widehat{g}_0(k)|d\tau\\
        &=-\int_0^t2\sum_{k\geq1}\frac{|k|^3}{|a_k|}\partial_\tau (e^{-|a_k| \tau})e^{\nu|k|\tau}|k||\widehat{g}_0(k)|d\tau,
    \end{aligned}
\end{equation*}
and integrating by parts gives 
\begin{equation*}
    \begin{aligned}
        \int_0^t\|e^{-\tau\mathcal{A}}g_0\|_{\dot{\mathcal{F}}_{\nu}^{4,1}}d\tau&=\int_0^t2\nu\sum_{k\geq1}\frac{|k|^3}{|a_k|}e^{-|a_k| \tau}e^{\nu|k|\tau}|k|^2|\widehat{g}_0(k)|d\tau\\
        &\quad-\Big[2\sum_{k\geq1}e^{-|a_k| \tau}e^{\nu|k|\tau}\frac{|k|^3}{|a_k|}|k||\widehat{g}_0(k)|\Big]_{\tau=0}^{\tau=t}.
    \end{aligned}
\end{equation*}
Therefore,
\begin{equation*}
    \begin{aligned}
        \int_0^t\|e^{-\tau\mathcal{A}}g_0\|_{\dot{\mathcal{F}}_{\nu}^{4,1}}d\tau&\leq C(\nu)\|g_0\|_{\dot{\mathcal{F}}^{1,1}},
    \end{aligned}
\end{equation*}
    where we have used that $|a_k|\simeq |k|^3$ and $\nu$ is taken small enough.
Thus the first estimate in \eqref{semig2} follows. Also, 
\begin{equation*}
    \begin{aligned}
        \int_0^t\|F(\tau)\|_{\dot{\mathcal{F}}_{\nu}^{4,1}}d\tau&=\int_0^t2\sum_{k\geq1} |k|^{4}e^{\nu|k|\tau}\int_0^\tau e^{-(\tau-w)|a_k|}|\widehat{g}(w,k)|dwd\tau\\
        &=2\sum_{k\geq1} \int_0^t\Big(\int_0^\tau|k|^3 e^{-(|a_k|-\nu|k|)(\tau-w)}|k|e^{\nu |k|w}|\widehat{g}(w,k)| dw\Big)d\tau\\
        &\leq C(\nu)2\sum_{k\geq1}\int_0^t |k|e^{\nu |k| \tau}|\widehat{g}(\tau,k)| d\tau\\
        &\leq C(\nu)\int_0^t \|g(\tau)\|_{\dot{\mathcal{F}}_{\nu}^{1,1}}d\tau,
        \end{aligned}
\end{equation*}
    where we have used again that $|a_k|\simeq |k|^3$ and Young's inequality for convolutions.
    
\end{proof}

\begin{lemma}[Interpolation]\label{interp_lem} Let $0 \leq s_1 \leq s_2$, $\theta\in[0,1]$ and $s= \theta s_{1} + (1-\theta)s_{2}$. Then,
    \begin{align}
        \label{Wienerinterpolation}
            \|f\|_{\fsonenu} &\leq \|f\|_{\dot{\mathcal{F}}^{s_{1},1}_\nu}^{\theta}\|f\|_{\dot{\mathcal{F}}^{s_{2},1}_\nu}^{1-\theta},\\
                \label{Wienerorder}
\|f\|_{\dot{\mathcal{F}}^{s_{1},1}_\nu}&\leq \|f\|_{\dot{\mathcal{F}}^{s_{2},1}_\nu}.    
    \end{align}
\end{lemma}
\begin{proof}
The interpolation inequality follows from Holder's inequality with $1/p = \theta$ and $1/q = 1- 1/p = 1-\theta$:
\begin{align*}
         \|f\|_{\fsonenu} &= \sum_{k\in\mathbb{Z}}e^{\nu|k|t}|k|^{s}|\hat{f}(k)|\\
         &=  \sum_{k\in\mathbb{Z}}e^{\nu|k|t \theta}|k|^{s_{1}\theta}|\hat{f}(k)|^{\theta}e^{\nu|k|t (1-\theta)}|k|^{s_{2}(1-\theta)}|\hat{f}(k)|^{1-\theta}\\
         &\leq \|f\|_{\dot{\mathcal{F}}^{s_{1},1}_\nu}^{\theta}\|f\|_{\dot{\mathcal{F}}^{s_{2},1}_\nu}^{1-\theta}.
    \end{align*}
    The second fact is a direct consequence of $|k|^{s_{1}}\leq |k|^{s_{2}}$ for all $|k|\geq 1$.
\end{proof}

\section{Estimates on General Nonlinear Terms}\label{lemmasection}

In this section, we will prove general estimates that will be applied in the later parts of the paper to bound the nonlinear terms appearing in both of the vorticity equation and evolution equation. 

We will first need the following lemma:

\begin{lemma}\label{lemmaI}
	Let $n\geq0$, $k$, $k_1,\dots, k_n$ be integers numbers. Then, the following bounds hold: 
	\begin{equation*}
	\begin{aligned}
	&\|I_1(k,k_1,\dots,k_n)\|_{L^\infty}=\|\int_{\mathbb{T}}\frac{\sin{(k\beta/2)}}{2\sin{(\beta/2)}}\prod_{j=0}^{n}\frac{\sin{(k_j\beta/2)}}{k_j\sin{(\beta/2)}}d\beta\|_{L^\infty}\leq 4,\\
	&\|I_2(k,k_1,\dots,k_n)\|_{L^\infty}=\|\int_{\mathbb{T}}\frac{\cos{(\beta/2)}\sin{(k\beta/2)}}{2\sin{(\beta/2)}}\prod_{j=0}^{n}\frac{\sin{(k_j\beta/2)}}{k_j\sin{(\beta/2)}}d\beta\|_{L^\infty}\leq \frac{10}{3}.
	\end{aligned}
	\end{equation*}
\end{lemma}
\begin{proof}

Both integrals are bounded in a similar manner. The bounds are optimal since they are attained for $n=0$ and $k=0$ in the first case, $k=3$ for $I_2$. We show the proof for $I_1$. First, notice that we can assume all $k_j\geq 1$, so we rewrite the quotients in the product as follows	
	\begin{equation*}
	\begin{aligned}
	\frac{\sin{(k_j\beta/2)}}{\sin{(\beta/2)}}&=\frac{e^{ik_j\beta/2}-e^{-ik_j\beta/2}}{e^{i\beta/2}-e^{-i\beta/2}}=\frac{e^{ik_j\beta/2}(1-e^{-ik_j\beta})}{e^{i\beta/2}(1-e^{-i\beta})}\\
	&=e^{i(k_j-1)\beta/2}\sum_{m=0}^{k_j-1}e^{-i\beta m}=\sum_{m=0}^{k_j-1}e^{i(-2m+k_j-1)\beta/2},
	\end{aligned}
	\end{equation*}
	so that
	\begin{equation*}
	\begin{aligned}
	\prod_{j=0}^n\frac{\sin{(k_j\beta/2)}}{k_j\sin{(\beta/2)}}&=\prod_{j=0}^n\frac{1}{k_j}\sum_{m_j=0}^{k_j-1}e^{j(-2m_j+k_j-1)\beta/2}\\
	&=\frac{1}{k_0k_1\dots k_n}\sum_{m_0=0}^{k_0-1}\dots\sum_{m_n=0}^{k_n-1}e^{i\beta/2(-2(m_0+\dots+m_n)+k_0+\dots+k_n-n-1)}.
	\end{aligned}
	\end{equation*}
	Denote
	\begin{equation*}
	A=-2(m_0+\dots+m_n)+k_0+\dots+k_n-n-1\in\mathbb{Z}.
	\end{equation*}
	Then, the expression of $I_1$ is reduced to
	\begin{equation*}
	\begin{aligned}
	I_1&=\frac{1}{k_0k_1\dots k_n}\sum_{m_0=0}^{k_0-1}\dots\sum_{m_n=0}^{k_n-1} \int_{\mathbb{T}} \frac{\sin{(k\beta/2)}}{2\sin{(\beta/2)}}e^{iA\beta/2}d\beta\\
	&=\frac{1}{k_0k_1\dots k_n}\sum_{m_0=0}^{k_0-1}\dots\sum_{m_n=0}^{k_n-1} \int_{\mathbb{T}} \frac{\sin{(k\beta/2)}\cos{(A\beta/2)}}{2\sin{(\beta/2)}}d\beta,
	\end{aligned}
	\end{equation*}
	which is bounded by
	\begin{equation*}
	|I_1|\leq \Big| \int_{\mathbb{T}} \frac{\sin{(k\beta/2)}\cos{(A\beta/2)}}{2\sin{(\beta/2)}}d\beta\Big|.
	\end{equation*}
	Consider this last integral as a function in $k, A\in\mathbb{Z}$. After a change of variables, we have to deal with the following integral
	\begin{equation*}
	I_1(k,A)=2\int_{0}^{\pi/2} \frac{\sin{(kx)}\cos{(Ax)}}{\sin{(x)}}dx.
	\end{equation*}
	We now proceed to compute it. First, $I_1(0,A)=0$ for all $A$. Now, for $A=0$, 
	\begin{equation*}
	\begin{aligned}
	I_1(k,0)&=2\int_{0}^{\pi/2}\! \frac{\sin{(kx)}}{\sin{(x)}}d\beta=2\int_{0}^{\pi/2}\!\frac{\sin{((k\!-\!1)x)}\cos{(x)}\!+\!\sin{(x)}\cos{((k\!-\!1)x)}}{\sin{(x)}}dx\\
	&=2\!\int_{0}^{\pi/2}\!\left(\frac{\sin{((k\!-\!2)x)}\cos^2{(x)}}{\sin{(x)}}+\cos{((k\!-\!2)x)}\cos{(x)}\!+\!\cos{((k\!-\!1)x)}\right)dx\\
	&=I_1(0,k-2)+2\int_0^{\pi/2}\cos{((k-1)x)}dx+2\frac{\sin{((k-1)\pi/2)}}{k-1}\\
	&=I_1(0,k-2)+4\frac{\sin{((k-1)\pi/2)}}{k-1}.
	\end{aligned}
	\end{equation*}
	Since $I_1(0,0)=0$ and $I_1(1,0)=\pi$,
	\begin{equation*}
	I_1(k,0)=\left\{ 
	\begin{aligned}
	&\sum_{j=1}^l 4\frac{(-1)^{j+1}}{2j-1}\hspace{0.5cm}\text{if } k=2l,\\
	&\pi \hspace{2.25cm}\text{if }k=2l+1,
	\end{aligned}\right.
	\end{equation*}
	so $|I_1(k,0)|\leq 4$ for all $k\in\mathbb{Z}$. Finally, when $k,A>0$,
	\begin{equation*}
	\begin{aligned}
	I_1(k,A)&=2\int_0^{\pi/2}\frac{\sin{((k\!-\!1)x)}\cos{(x)}\cos{(Ax)}}{\sin{(x)}}dx\!+\!2\int_0^{\pi/2}\!\!\cos{((k\!-\!1)x)}\cos{(Ax)}dx\\
	&=2\int_0^{\pi/2}\frac{\sin{((k-1)x)}\cos{((A+1)x)}}{\sin{(x)}}dx+2\int_0^{\pi/2}\!\!\sin{((k-1)x)}\sin{(Ax)}dx\\
	&\quad+2\int_0^{\pi/2}\cos{((k-1)x)}\cos{(Ax)}dx\\
	&=I_1(k-1,A+1)+2\frac{\sin{((A-k+1)\frac\pi2)}}{A-k+1}.
	\end{aligned}
	\end{equation*}
	Therefore, 
	\begin{equation*}
	\begin{aligned}
	I_1(k-1,A+1)=I_1(k,A)-2\frac{\sin{((A-k+1)\frac\pi2)}}{A-k+1},
	\end{aligned}
	\end{equation*}
	which yields
	\begin{equation}\label{aux1}
	\begin{aligned}
	I_1(k,A)&=I_1(k+1,A-1)-2\frac{\sin{((A-k-1)\frac\pi2)}}{A-k-1}=\dots\\
	&=I_1(k+A,0)-2\sum_{n=0}^{A-1}\frac{\sin{((A-k-2n-1)\frac\pi2)}}{A-k-2n-1}.
	\end{aligned}
	\end{equation}
	On the other hand,
	\begin{equation*}
	\begin{aligned}
	\int_0^{\pi/2}\frac{\sin{(kx)}\cos{(Ax)}}{\sin{(x)}}dx=\int_0^{\pi/2}\frac{\sin{((A+k)x)}}{\sin{(x)}}dx-\int_0^{\pi/2}\frac{\sin{(Ax)}\cos{(kx)}}{\sin{(x)}}dx,
	\end{aligned}
	\end{equation*}
	which means that
	\begin{equation*}
	I_1(k,A)=I_1(k+A,0)-I_1(A,k).
	\end{equation*}
	Combining this with \eqref{aux1}, we obtain that
	\begin{equation*}
	I_1(A,k)=2\sum_{n=0}^{A-1}\frac{\sin{((A-k-2n-1)\frac\pi2)}}{A-k-2n-1},
	\end{equation*}
	so, exchanging the roles of $k$ and $A$,
	\begin{equation*}
	I_1(k,A)=2\sum_{n=0}^{k-1}\frac{\sin{((k-A-2n-1)\frac\pi2)}}{k-A-2n-1},
	\end{equation*}
	we finally have that
	\begin{equation*}
	I_1(k,A)=\left\{
	\begin{aligned}
	&\pi \hspace{1cm}\text{ if }k-A \text{ is odd},\\
	&2 \sum_{n=0}^{k-1}\frac{(-1)^{l-n+1}}{2(l-n)-1} \text{ if } k-A=2l.
	\end{aligned}\right.
	\end{equation*}
	Therefore, we conclude that $|I_1|\leq |I_1(k,A)|\leq 4$ for all $k,A\in \mathbb{Z}$.
	
\end{proof}

We can apply this lemma to bound each frequency of nonlinear terms. For simplicity, we introduce the following notation for convolutions of $n$ functions:
\begin{equation*}
    f_{1}\ast f_{2} \ast \cdots \ast f_{n} = \ast_{j=1}^{n} f_{j},
\end{equation*}
regardless of the domain of all the functions.
\begin{lemma}\label{nonlinearfourierlemma}
Consider the two functions of the form
\begin{equation*}
   I(\alpha) = \textup{\pv} \int_{\mathbb{T}}\prod_{j=1}^{l}f_{j}(\alpha-\beta) \prod_{j=l+1}^{n}\Delta_{\beta}f_{j}(\alpha) \frac{d\beta}{2\sin(\beta/2)},
\end{equation*}
and
\begin{equation*}
   J(\alpha) = \textup{\pv} \int_{\mathbb{T}}\prod_{j=1}^{l}f_{j}(\alpha-\beta) \prod_{j=l+1}^{n}\Delta_{\beta}f_{j}(\alpha) \frac{d\beta}{2\tan(\beta/2)}.
\end{equation*}
Then, the Fourier transform of these functions obey the following bounds:
\begin{equation*}
    |\hat{I}(k)|\leq 4 \left(\ast_{j=1}^{l} |\hat{f_{j}}(k)|\right) \ast \left(\ast_{j=l+1}^{j=n} |k||\hat{f_{j}}(k)|\right),
\end{equation*}
and
\begin{equation*}
    |\hat{J}(k)| \leq \frac{10}{3}\left(\ast_{j=1}^{l} |\hat{f_{j}}(k)|\right) \ast \left(\ast_{j=l+1}^{j=n} |k||\hat{f_{j}}(k)|\right).
\end{equation*}
\end{lemma}

\begin{proof}
First, from the definition of $\Delta_{\beta}f(\alpha)$, we obtain that
    \begin{equation}\label{deltabetafourier}
\widehat{\Delta_{\beta}f}(\alpha)=\frac{1-e^{-ik\beta}}{2\sin{\left(\beta/2\right)}}\hat{f}(k)=m(k,\beta)\hat{f}(k),
\end{equation}
where
\begin{equation*}
m(k,\beta)=\frac{1-e^{-ik\beta/2}e^{-ik\beta/2}}{2\sin{\left(\beta/2\right)}}=\frac{e^{ik\beta/2}-e^{-ik\beta/2}}{2\sin{\left(\beta/2\right)}}e^{-ik\beta/2}=ik\frac{\sin{\left(k\beta/2\right)}}{k\sin{\left(\beta/2\right)}}e^{-ik\beta/2}.
\end{equation*}
Hence, we obtain that
\begin{align*}
    \hat{I}(k_{0}) &= \textup{\pv} \int_{\mathbb{T}}\left(\ast_{j=1}^{j=l}\hat{f}_{j}(k)e^{-ik\beta}\right) \ast \left(\ast_{j=l+1}^{j=n}m(k,\beta)\hat{f}_{j}(k) \right)(k_0)\frac{d\beta}{2\sin(\beta/2)}\\
    &= \textup{\pv} \int_{\mathbb{T}} \sum_{k_{j}\in\mathbb{Z}} \prod_{j=1}^{n} \hat{f_{j}}(k_{j-1}-k_{j}) \cdot \prod_{j=1}^{l}e^{-i(k_{j-1}-k_{j})\beta}\\
    &\hspace{0.75in}\cdot \prod_{j=l+1}^{n}m(k_{j-1}-k_{j},\beta) \cdot m(k_{n},\beta)\hat{f_n}(k_{n}) \ \frac{d\beta}{2\sin(\beta/2)}.
\end{align*}
Note that in the first line above, $k_{0}$ is the input into the convolution of the $n$ functions in that line. Next, we compute that
\begin{align*}
\prod_{j=1}^{l}e^{-i(k_{j-1}-k_{j})\beta}\cdot \prod_{j=l+1}^{n}m(k_{j-1}-k_{j},\beta) \cdot m(k_{n},\beta) &= e^{iA} \prod_{j=l+1}^{n}(k_{j-1}-k_{j}) \cdot k_{n} \cdot \frac{\sin(k_{n}\beta/2)}{k_{n}\sin(\beta/2)} \\
 &\hspace{1cm} \cdot \prod_{j=l+1}^{n} \frac{\sin((k_{j-1}-k_{j})\beta/2)}{(k_{j-1}-k_{j})\sin(\beta/2)},
\end{align*}
where $A$ is a sum of terms involving $\beta$ and $k_j$ for $j=1,\ldots,n$. Hence, for $I_{1} = I_{1}(k, k_1 - k_0, \ldots, k_{n-1}-k_n, k_n)$, we can use Lemma \ref{lemmaI} to obtain that
\begin{align*}
    |\hat{I}(k_{0})| &\leq \sum_{k_{j}\in\mathbb{Z}} \prod_{j=1}^{l}|\hat{f_{j}}(k_{j-1}-k_{j})|\prod_{j=l+1}^{n} |k_{j-1}-k_{j}||\hat{f_{j}}(k_{j-1}-k_{j})|  \cdot |k_n||\hat{f_n}(k_n)|\cdot \|I_{1}\|_{L^{\infty}}\\
    &\leq 4 \left(\left(\ast_{j=1}^{l} |\hat{f_{j}}(k)|\right) \ast \left(\ast_{j=l+1}^{j=n} |k||\hat{f_{j}}(k)|\right)\right)(k_{0}).
\end{align*}
This concludes the proof of the bound on $\hat{I}$. The proof for $\hat{J}$ is the same.
\end{proof}

Finally, as a corollary to Lemma \ref{nonlinearfourierlemma}, we obtain that
\begin{lemma}\label{nonlinearlemma}[Nonlinear estimate]
For the functions $I(\alpha)$ and $J(\alpha)$ described in Lemma \ref{nonlinearfourierlemma}, we have the bound
\begin{equation*}
\|I,J\|_{\fsonenu} \lesssim \sum_{i=1}^{l} \|f_{i}\|_{\fsonenu} \prod_{\substack{j=1\\ j\neq i}}^{l} \|f_{j}\|_{\fzeronenu}\prod_{j=l+1}^{n}\|f_{i}\|_{\foneonenu} + \sum_{i=l+1}^{n} \|f_{i}\|_{\fsplusonenu} \prod_{j=1}^{l}\|f_{j}\|_{\fzeronenu}\prod_{\substack{j=l+1\\ j\neq i}}^{n}\|f_{i}\|_{\foneonenu},
\end{equation*}
for any $0<s\leq 1$.
For $s=0$, we have
\begin{equation*}
\|I,J\|_{\fzeronenu} \lesssim \prod_{j=1}^{l} \|f_{j}\|_{\fzeronenu}\prod_{j=l+1}^{n}\|f_{i}\|_{\foneonenu}.
\end{equation*}
\end{lemma}
\begin{proof}
By Lemma \ref{nonlinearfourierlemma}, we have that
\begin{align*}
\|I\|_{\fsonenu} &\leq 4 \Big\|\left(\ast_{j=1}^{l} |\hat{f_{j}}(k)|\right) \ast \left(\ast_{j=l+1}^{j=n} |k||\hat{f_{j}}(k)|\right) \Big\|_{\fsonenu}\\
&= 4\Big\|e^{\nu t |k|} |k|^{s}\left(\ast_{j=1}^{l} |\hat{f_{j}}(k)|\right) \ast \left(\ast_{j=l+1}^{j=n} |k||\hat{f_{j}}(k)|\right) \Big\|_{\ell^{1}}.
\end{align*}
By the triangle inequality
\begin{equation*}
|k|^{s} \leq |k-k_{1}|^{s} + |k_{1}-k_{2}|^{s} + \ldots + |k_{2n-1}-k_{2n}|^{s} + |k_{2n}|^{s},
\end{equation*}
for any $0<s\leq 1$ and hence
\begin{equation*}
e^{t\nu |k|} \leq e^{t\nu |k-k_{1}|}\cdot \prod_{j=1}^{n}e^{t\nu |k_{j-1}-k_{j}|} \cdot e^{t\nu |k_{2n}|}.
\end{equation*}
Therefore,
\begin{align*}
\|I\|_{\fsonenu} &\leq 4\Big\| |k|^{s}\left(\ast_{j=1}^{l} e^{t\nu |k|}|\hat{f_{j}}(k)|\right) \ast \left(\ast_{j=l+1}^{j=n} |k|e^{t\nu |k|}|\hat{f_{j}}(k)|\right) \Big\|_{\ell^{1}}\\
&\leq 4 \sum_{i=1}^{l} \Big\|e^{t\nu |k|}|k|^{s}|\hat{f_{i}}(k)| \ast \left(\ast_{j\neq i, j=1}^{j=l} e^{t\nu |k|}|\hat{f_{j}}(k)|\right) \ast \left(\ast_{j=l+1}^{j=n} |k|e^{t\nu |k|}|\hat{f_{j}}(k)|\right) \Big\|_{\ell^{1}}\\
&\hspace{0.25in} + 4 \sum_{i=l+1}^{n} \Big\|\left(\ast_{j=1}^{j=l} e^{t\nu |k|}|\hat{f_{j}}(k)|\right) \ast \left(\ast_{j\neq i, j=l+1}^{j=n} |k|e^{t\nu |k|}|\hat{f_{j}}(k)|\right) \ast e^{t\nu |k|}|k|^{s+1}|\hat{f_{i}}(k)|\Big\|_{\ell^{1}}.
\end{align*}
Applying Young's inequality for convolutions in $\ell^{1}$ concludes the proof.

\end{proof}

\section{Vorticity Estimates}

For $f\in\mathcal{F}^{1,1}$,
recall the operator $\mathcal{D}[f]$ defined on the vorticity function $\tilde{\omega}$ in \eqref{vorticity_tilde}. Here we rewrite it as a linear operator on the space $Z$ \eqref{Z_space} of functions with mean zero. For $g\in Z$,
$\mathcal{D}[f](g)$ is given by
\begin{equation*}
\begin{aligned}
\mathcal{D}[f](g)(\alpha)&\!=\frac{1}{2\pi}\pv\!\int_{\mathbb{T}}\!\!\frac{(1\!+\!f(\alpha))(1\!+\!f(\alpha\!-\!\beta))\sin{\left(\beta/2\right)}\!+\!(1\!+\!f(\alpha))\Delta_{\beta} f(\alpha)}{(\Delta_\beta f(\alpha))^2\!+\!(1\!+\!f(\alpha))(1\!+\!f(\alpha\!-\!\beta))}\frac{g(\alpha\!-\!\beta)}{2\sin{\left(\beta/2\right)}}d\beta\\
&\quad-\frac{1}{2\pi}\pv\int_{\mathbb{T}}\frac{(1+f(\alpha-\beta))\partial_{\alpha}f(\alpha)\cos{\left(\beta/2\right)}}{(\Delta_\beta f(\alpha))^2+(1+f(\alpha))(1+f(\alpha-\beta))}\frac{g(\alpha-\beta)}{2\sin{\left(\beta/2\right)}}d\beta.
\end{aligned}
\end{equation*}
When $\|f\|_{W^{1,\infty}}$ is small enough, we can expand the above expression via a Taylor series for $f$. For example, in the first term of the first integral of $\mathcal{D}[f]$, we have
\begin{align*}
\frac{(1+f(\alpha))(1+f(\alpha-\beta))}{(\Delta_\beta f(\alpha))^2+(1+f(\alpha))(1+f(\alpha-\beta))}&=\frac{1}{1 + \frac{\Delta_\beta f(\alpha)^2}{(1+f(\alpha))(1+f(\alpha-\beta))}}\\
&=\sum_{n=0}^{\infty}(-1)^n \frac{\Delta_\beta f(\alpha)^{2n}}{(1+f(\alpha))^{n}(1+f(\alpha-\beta))^{n}}.
\end{align*}
Expanding the entire kernel in this manner, we can right
\begin{equation*}
   \mathcal{D}[f](g) = \mathcal{D}_0(g)+ \mathcal{D}_1[f](g) + \mathcal{D}_{\geq 2}[f](g),
\end{equation*}
where $D_{0}(g)$ is the zero order terms in $f$ in the Taylor expansion of the kernel, i.e. no dependence on $f$, and $\mathcal{D}_{1}$ are the first order terms, i.e. linear with respect to $f$. In the space of functions $Z$, the order zero term integrates away:
\begin{equation*}
\mathcal{D}_{0}(g) = \int_{\mathbb{T}}  \sin(\beta/2)\frac{g(\alpha-\beta)}{\sin(\beta/2)} d\beta = 0.
\end{equation*}
Hence:
\begin{equation*}
   \mathcal{D}[f](g) = \mathcal{D}_1[f](g) + \mathcal{D}_{\geq 2}[f](g).
\end{equation*}
The first order terms in the kernel of the operator give:
\begin{align*}
    \mathcal{D}_1[f](g) &= \frac{1}{2\pi}\pv\int_{\mathbb{T}}\left(\Delta_\beta f(\alpha)-\partial_{\alpha}f(\alpha)\cos(\beta/2)\right)\frac{g(\alpha-\beta)}{2\sin{\left(\beta/2\right)}}d\beta.
\end{align*}
Both $\mathcal{D}_{1}$ and the higher order terms $\mathcal{D}_{\geq 2}$ are of the form $I$ or $J$ found in Lemma \ref{nonlinearfourierlemma}. Hence, as a direct corollary of Lemma \ref{nonlinearlemma}, we obtain:
\begin{prop}\label{dnorm}
Let $\varepsilon>0$ be a small enough constant, $\nu\geq0$, and $f\in\mathcal{F}^{1,1}_\nu$ such that $\|f\|_{\mathcal{F}^{1,1}_\nu}\leq \varepsilon.$
Then,  the operator $\mathcal{D}[f]$ is a bounded linear operator in $Z_\nu$ \eqref{Z_space}: 
\begin{equation*}
\mathcal{D}[f]: Z_\nu \rightarrow Z_\nu,
\end{equation*}
which satisfies the bound
\begin{align*}
    \|\mathcal{D}[f]\|_{Z_\nu \rightarrow Z_\nu} &\lesssim_1 \|f\|_{\mathcal{F}^{1,1}_\nu}.
\end{align*}
\end{prop}
That $\mathcal{D}[f](g)$ has mean zero when $g$ does follows from the fact that $\mathcal{D}[f](g)$ can be written as a derivative. In fact, for $g=\partial_\alpha h$,  one can check from the expression of $\mathcal{D}[f](g)$ in terms of $z$ \eqref{Dz} that
\begin{equation*}
\begin{aligned}
\mathcal{D}[z](g)(\alpha,t)&=-BR[z](g)(\alpha,t)\cdot\partial_{\alpha}z(\alpha,t)\\
&=\frac1{2\pi}\partial_\alpha\Big( \pv\int_{\mathbb{T}}
\frac{(z(\alpha,t)\!-\!z(\beta,t))\cdot\partial_{\alpha}z(\beta,t)^\perp }{|z(\alpha,t)-z(\beta,t)|^2}\Omega(\beta,t)d\beta\Big).
\end{aligned}
\end{equation*}
We now rewrite the vorticity equation \eqref{vorticity_tilde} in the more general operator form:
\begin{equation}\label{vorticityoperator}
g(\alpha)=2 A_\mu \mathcal{D}[f](g)(\alpha)+F(\alpha)
\end{equation}
As a direct consequence of Proposition \ref{dnorm}, we have the following lemma.
\begin{prop}[Vorticity equation]\label{propvor2}
If $F\in Z_\nu$ and  $\|f\|_{\mathcal{F}^{1,1}_\nu} < \epsilon$ such that
$$2A_{\mu}\|\mathcal{D}[f]\|_{Z_\nu \rightarrow Z_\nu} < 1,$$
then the solution to \eqref{vorticityoperator} in the function space $Z_\nu$ is given by the Neumann series
\begin{equation}\label{neumann}
g = \sum_{n=0}^{\infty} 2^{n} A_{\mu}^{n} \mathcal{D}[f]^{n}(F).
\end{equation}
Moreover,
\begin{equation}\label{gnorm}
\|g\|_{\dot{\mathcal{F}}^{0,1}_\nu} \lesssim \Big(\sum_{n=0}^{\infty}2^{n}A_{\mu}^{n}\|f\|_{\mathcal{F}^{1,1}_\nu}^{n}\Big) \|F\|_{\dot{\mathcal{F}}^{0,1}_\nu}.
\end{equation}
\end{prop}
\begin{prop}\label{gfoneone}
If $F\in Z_\nu^{1,1}$ and $\|f\|_{\mathcal{F}^{1,1}_\nu} < \epsilon,$ then any solution to \eqref{vorticityoperator} satisfies the estimate
\begin{equation*}
\|g\|_{\foneonenu} \lesssim_{1} \|f\|_{\ftwoonenu}\|F\|_{\dot{\mathcal{F}}^{0,1}_\nu} + \|F\|_{\foneonenu}.
\end{equation*}
\end{prop}
\begin{proof}
We will use Lemma \ref{nonlinearlemma} for $s=1$ to obtain that
\begin{equation*}
\|g\|_{\foneonenu}\lesssim_{1} \|f\|_{\foneonenu}\|g\|_{\foneonenu} + \|f\|_{\ftwoonenu}\|g\|_{\dot{\mathcal{F}}^{0,1}_\nu} + \|F\|_{\foneonenu}
\end{equation*}
Hence, taking $\|f\|_{\foneonenu} < \epsilon$ for $\epsilon$ small enough and using \eqref{gnorm}, we obtain  the desired estimate.

\end{proof}

\begin{cor}[Vorticity Estimate]\label{vorticity_cor} 
Let $f\in\mathcal{F}^{4,1}_\nu$ with $\|f\|_{\mathcal{F}^{1,1}}<\varepsilon$ and denote
\begin{equation}\label{vort_F}
    \begin{aligned}
        F[f](\alpha):=2\partial_{\alpha}K(f)(\alpha)-2A_{\rho\sigma} \partial_\alpha\big((1+f(\alpha))\sin{\alpha}\big),
    \end{aligned}
\end{equation}
where $K(f)$ is defined in \eqref{curvaturef}. Then, if we define $\tilde{\omega}$ by the Neumann series
\begin{equation*}
\tilde{\omega} = \sum_{n=0}^{\infty} 2^{n} A_{\mu}^{n} \mathcal{D}[f]^{n}(F[f]),
\end{equation*}
it holds that $\tilde{\omega}\in\mathcal{F}^{1,1}_\nu$, it solves the equation \eqref{vorticityoperator} with $F$ given by \eqref{vort_F} and, moreover,
\begin{equation*}
\|\tilde{w}_{\geq1}\|_{\dot{\mathcal{F}}^{0,1}_\nu}\lesssim_1 |A_{\rho\sigma}|(1+A_\mu)\|f\|_{\mathcal{F}^{1,1}_\nu}+\|f\|_{\dot{\mathcal{F}}^{3,1}_\nu},
\end{equation*}
and
\begin{equation*}
\|\tilde{\omega}_{\geq 1}\|_{\foneonenu} \lesssim_{1} |A_{\rho\sigma}|(1+A_\mu)\|f\|_{\mathcal{F}^{2,1}_\nu} + \|f\|_{\ffourone_{\nu}}. 
\end{equation*}
\end{cor}
\begin{proof}
It only remains to prove the estimates. We write
\begin{equation*}
    \begin{aligned}
        \tilde{\omega}_{\geq1} =F[f]-F_0[f]+\sum_{n=1}^{\infty} 2^{n} A_{\mu}^{n} \mathcal{D}[f]^{n}(F[f]),
    \end{aligned}
\end{equation*}
thus, using  Propositions \ref{dnorm}-\ref{gfoneone},
\begin{equation*}
    \begin{aligned}
        \|\tilde{\omega}_{\geq1}\|_{\dot{\mathcal{F}}^{0,1}_\nu} &\lesssim_1 \|F_{\geq1}[f]\|_{\dot{\mathcal{F}}^{0,1}_\nu}+ A_\mu\|f\|_{\mathcal{F}^{1,1}_\nu}\|F_0[f]\|_{\mathcal{F}^{0,1}_\nu},
    \end{aligned}
\end{equation*}
and
\begin{equation*}
    \begin{aligned}
        \|\tilde{\omega}_{\geq1}\|_{\dot{\mathcal{F}}^{1,1}_\nu} &\lesssim_1 \|F_{\geq1}[f]\|_{\dot{\mathcal{F}}^{1,1}_\nu}+ A_\mu\|f\|_{\mathcal{F}^{2,1}_\nu}\|F_0[f]\|_{\mathcal{F}^{0,1}_\nu}.
    \end{aligned}
\end{equation*}
Finally, the estimates 
\begin{equation*}
    \begin{aligned}
        \|F_0[f]\|_{\mathcal{F}^{0,1}_\nu}&\lesssim |A_{\rho\sigma}|,\\
        \|F_{\geq1}[f]\|_{\dot{\mathcal{F}}^{0,1}_\nu}&\lesssim_1 \|f\|_{\dot{\mathcal{F}}^{3,1}_\nu}+|A_{\rho\sigma}|\|f\|_{\mathcal{F}^{1,1}_\nu},\\
        \|F_{\geq1}[f]\|_{\dot{\mathcal{F}}^{1,1}_\nu}&\lesssim_1 \|f\|_{\dot{\mathcal{F}}^{4,1}_\nu}+|A_{\rho\sigma}|\|f\|_{\mathcal{F}^{2,1}_\nu},
    \end{aligned}
\end{equation*}
allow to conclude the proof.

\end{proof}

\section*{Acknowledgements}

FG and EGJ were partially supported by the AEI grants PID2020-114703GB-I00, PID2022-140494NA-I00 and RED2022-134784-T funded by MCIN/AEI/10.13039/501100011033. FG was partially supported by the AEI grant EUR2020-112271, by the Fundacion de Investigacion de la Universidad de Sevilla through the grant FIUS23/0207 and acknowledges support from IMAG, funded by MICINN through the Maria de Maeztu
Excellence Grant CEX2020-001105-M/AEI/10.13039/501100011033. EGJ was partially supported by the RYC 2021-032877 research grant, the AEI grant PID2021-125021NA-I00 and the AGAUR grant 2021-SGR-0087. RMS was partially supported by the NSF grant DMS-2055271 (USA). NP is partially supported by the AMS-Simons PUI faculty grant.



\bibliographystyle{amsplain2linkLink}
\bibliography{references}	

\vspace{1cm}

\end{document}